\documentclass{amsart}

\usepackage{graphics}
\usepackage{amsmath,amsfonts,latexsym,amssymb}
\usepackage{comment}
\usepackage{xcolor}
\usepackage{fullpage}

\usepackage[english]{babel}
\usepackage{amsthm}
\usepackage{amstext}
\usepackage{mathrsfs}
\usepackage{pgf,tikz}
\usepackage{ulem}
\usepackage{amstext} 
\usepackage{array}   
\newcolumntype{C}{>{$}c<{$}} 
\definecolor{uququq}{rgb}{0.25,0.25,0.25}
\newtheorem{thm}{Theorem}[section]
\newtheorem{cor}[thm]{Corollary}
\newtheorem{lem}[thm]{Lemma}
\newtheorem{prop}[thm]{Proposition}
\usetikzlibrary{arrows}
\theoremstyle{definition}
\newtheorem{defn}[thm]{Definition}
\theoremstyle{definition}
\newtheorem{rem}[thm]{Remark}
\theoremstyle{definition}
\newtheorem{ex}[thm]{Example}

\newcommand{\Z}{\mathbb{Z}}

\newcommand{\D}{\mathbb{D}}
\newcommand{\F}{\mathbb{F}}
\def\H{\mathrm{H}}
\def\D{\mathrm{D}}

\begin{document}

\title{Heffter Spaces}

\author{Marco Buratti}
\address{SBAI - Sapienza Universit\`a di Roma, Via Antonio Scarpa 16, I-00161 Roma, Italy}
\email{marco.buratti@uniroma1.it}

\author{Anita Pasotti}
\address{DICATAM - Sez. Matematica, Universit\`a degli Studi di Brescia, Via
Branze 43, I-25123 Brescia, Italy}
\email{anita.pasotti@unibs.it}
%
%

\subjclass[2010]{}

\maketitle

\begin{abstract}
The notion of a Heffter array, which received much attention in the last decade, is
equivalent to a pair of orthogonal Heffter systems.
In this paper we study the existence problem of a set of $r$ mutually orthogonal
Heffter systems for any $r$.
Such a set is equivalent to a resolvable partial linear space of degree $r$ whose parallel classes are  Heffter systems:
this is a new combinatorial design that we call a {\it Heffter space}.
We present a series of direct constructions of Heffter spaces with odd block size and
arbitrarily large degree $r$ obtained with the crucial use of finite fields.
Among the applications we establish, in particular, that if
$q=2kw+1$ is a prime power with $kw$ odd and $k\geq 3$, then
there are at least $\lceil{w\over4k^4}\rceil$ mutually orthogonal $k$-cycle systems of order $q$.
\end{abstract}

\section{Introduction}
We introduce {\it Heffter spaces}, a new combinatorial design generalizing the well-known notion of a {\it Heffter array}.
Whereas a Heffter array gives a pair of orthogonal Heffter systems, more generally a {\it Heffter space of degree $r$} gives
a set of $r$ mutually orthogonal Heffter systems.
We first need to recall some basic definitions and give some new ones.

The {\it patterned starter} of a finite group $G$ of odd order is the set of all pairs $\{g,-g\}$ with $g\in G\setminus\{0\}$.
Any complete system of representatives for the patterned starter of $G$ is said to be a {\it half-set} of $G$.

 \begin{defn}\label{HeffterSystem}
 Let $V$ be a half-set of an abelian group $G$ of order $2v+1\geq7$. A $(v,k)$ {\it Heffter system}
 on $V$ is a partition $\mathcal P$ of $V$ into zero-sum parts, called {\it blocks}, of size $k$. The size $v$ of $V$ will be called
 the {\it order} of $\mathcal P$. By saying that a Heffter system is {\it over a group $G$} we mean that it is on an unspecified
 half-set of $G$.
 \end{defn}

If $G=\Z_{2v+1}$, in the classic literature  one usually speaks of a $\D(2v+1, k)$, see \cite{A}.
We have changed the notation in order to be consistent with the one about Heffter spaces that we will introduce later on.
It could be helpful to keep in mind that if $\mathcal P$ is a Heffter system on $V$, then $V$ is certainly zero-sum since it is
partitioned by its blocks which are, by definition, zero-sum.

\begin{defn}\label{simple}
Throughout the paper, speaking of a {\it simple} $k$-subset $B$ of an additive group $G$ we tacitly assume that $B$ is an {\it ordered}
$k$-subset $\{b_0,b_1,\dots,b_{k-1}\}$ of $G$ such that the $k$-sequence of its partial sums $(c_0,c_1,\dots,c_{k-1})$ defined by
$\displaystyle c_i=\sum_{j=0}^i b_j$ for $0\leq i\leq k-1$, does not have repeated elements.

Speaking of a {\it simple} family ${\mathcal F}$ of subsets of $G$ we mean that every member of $\mathcal F$
is simple in the sense said above; thus it is assumed that every $B\in{\mathcal F}$ has been assigned an order
satisfying the above property.

In particular, a $(v,k)$  Heffter system ${\mathcal P}=\{B_1,\dots,B_n\}$ will be called simple
if every $B_i$ has been assigned an order for which its partial sums are all distinct.
\end{defn}

It has been conjectured that the blocks of {\it any} $(v,k)$ Heffter system over $\Z_{2v+1}$
can be suitably ordered to get a simple Heffter system. The conjecture is trivially true for $k\in\{3,4,5\}$.
It is also true for $6\leq k\leq10$ (see \cite{AL,CMPPSums}).

The existence of a simple $(v,k)$ Heffter system over $\Z_{2v+1}$ is implicitly assured for any admissible $v$, namely for any $v\geq3$ divisible by $k$,
by the existence of a cyclic $k$-cycle decomposition of $K_{2v+1}$ established independently in \cite{BGL,BurDel}.
Indeed, as it will be recalled in detail in the last section, such a decomposition is equivalent to a
simple $(v,k)$ Heffter system.



\begin{ex}\label{ex:D204}
Here is a half-set of $\Z_{41}$ $$V=\{1, 2, 3, 4, 5, 6, 7, 9, 10, 11, 13, 15, 19, 21, 23, 24, 25, 27, 29, 33\}$$
and here are three $(20,4)$ Heffter systems on $V$
$${\mathcal P}_1=\{\{1,3,4,33\}, \{2,5,13,21\}, \{6,23,24,29\},\{7,9,10,15\},\{11,19,25,27\}\};$$
$${\mathcal P}_2=\{\{1,2,9,29\}, \{3,6,13,19\}, \{4,5,7,25\},\{10,21,24,27\},\{11,15,23,33\}\};$$
$${\mathcal P}_3=\{\{1,6,7,27\}, \{2,4,11,24\}, \{3,5,10,23\}, \{9,19,21,33\},\{13,15,25,29\}\}.$$
\end{ex}

\begin{defn}
Two Heffter systems ${\mathcal P}$ and ${\mathcal Q}$ on the same half-set are {\it orthogonal} if every block of $\mathcal P$
intersects every block of $\mathcal Q$  in at most one element, see \cite{A}.
Speaking of a $(v,\{k_1,\dots,k_r\})$-MOHS we mean a set $\{{\mathcal P}_1, \dots, {\mathcal P}_r\}$ of mutually orthogonal Heffter systems
of order $v$ where the blocks of ${\mathcal P}_i$ have size $k_i$ for $1\leq i\leq r$. Sometimes the multiset of block sizes will be denoted in
exponential notation $\{k_1^{m_1},\dots,k_s^{m_s}\}$ but we will write $(v,k;r)$-MOHS rather than $(v,\{k^r\})$-MOHS.
\end{defn}

\begin{ex}\label{ex:D414bis}
One can check that the three $(20,4)$ Heffter systems ${\mathcal P}_1$, ${\mathcal P}_2$,  ${\mathcal P}_3$ of Example \ref{ex:D204} are
mutually orthogonal, i.e., $\{{\mathcal P}_1,{\mathcal P}_2,{\mathcal P}_3\}$ is a $(20,4;3)$-MOHS.
\end{ex}
Let $\{{\mathcal P},{\mathcal Q}\}$ be a $(v,\{h,k\})$-MOHS, i.e., a pair of orthogonal Heffter systems of order $v$ where the blocks of
${\mathcal P}$ and ${\mathcal Q}$ have sizes $h$ and $k$, respectively. Set
${\mathcal P}=\{B_1,\dots,B_m\}$ and ${\mathcal Q}=\{B'_1,\dots,B'_n\}$
so that we have $m={v\over h}$ and $n={v\over k}$. Then $\{{\mathcal P},{\mathcal Q}\}$ can be efficiently
displayed in an $m\times n$ partially filled matrix $A$ whose cell $a_{i,j}$ is empty if $B_i$ and $B'_j$ are disjoint or contains their
common element otherwise. Thus $A$ satisfies the following conditions:
\begin{itemize}
  \item[{\rm (a)}] each row has exactly $h$ filled cells;
    \item[{\rm (b)}] each column has exactly $k$ filled cells;
\item[{\rm (c)}]  the entries of the filled cells form a half-set of an abelian group $G$;
\item[{\rm (d)}] every row and every column is zero-sum in $G$.
\end{itemize}


An $m\times n$ matrix $A$ with the above properties is said to be a {\it Heffter array} $\H(m,n;h,k)$.
Conversely, it is easy to see that any $\H(m,n;h,k)$ gives rise to a $(nk,\{h,k\})$-MOHS.
A {\it square} Heffter array is an $\H(n,n;k,k)$ and is denoted by H$(n;k)$.

\begin{ex}\label{ex54}
The following three $\H(5;4)$ display the mutual orthogonality between the three Heffter systems in Example \ref{ex:D204}.
$$
\begin{array}{|r|r|r|r|r|} \hline
1 & 3 & 4 &  & 33 \\ \hline
2 & 13 & 5 & 21 & \\ \hline
29 & 6 & & 24 & 23\\ \hline
9 &  & 7 & 10 & 15\\ \hline
 & 19 & 25 & 27 & 11\\ \hline
\end{array}
\quad\quad\quad\quad\quad\quad
\begin{array}{|r|r|r|r|r|} \hline
1 & 4 & 3 & 33 &  \\ \hline
 & 2 & 5 & 21 & 13 \\ \hline
6 & 24 & 23 &  & 29\\ \hline
7 &  & 10 & 9 & 15\\ \hline
 27 & 11 &  & 19 & 25\\ \hline
\end{array}
\quad\quad\quad\quad\quad\quad
\begin{array}{|r|r|r|r|r|} \hline
1 & 2 &  & 9 & 29 \\ \hline
6 &  & 3 & 19 & 13 \\ \hline
7 & 4 & 5 &  & 25\\ \hline
27 & 24 & 10 & 21 & \\ \hline
 & 11 & 23 & 33 & 15\\ \hline
\end{array}
$$
\end{ex}

Heffter arrays have been introduced in 2015 by Archdeacon \cite{A} who established in this way
an interesting link between combinatorial designs and topological graph theory.
After that, Heffter arrays have been studied by many authors. We refer to \cite{DP}
for an extensive survey on their results. Here, we just recall that the existence problem for the square
Heffter arrays has been completely solved.
\begin{thm}\label{thm:Heffter}{\rm \cite{ADDY,CDDY,DW}}
There exists an $\H(n; k)$ if and only if $n \geq k \geq  3$.

\end{thm}




\medskip
Let us open a parenthesis to recall some definitions from classic design theory \cite{BJL,CD}.

A {\it partial linear space} (PLS) is a pair $(V,{\mathcal B})$ where $V$ is a set of {\it points} and $\mathcal B$
is a set of non-empty subsets ({\it blocks} or {\it lines}) of $V$ with the property that any 2-subset of $V$ is contained in at most one block.
Two distinct points are said to be {\it collinear} if there is a block containing them.

A PLS where every two distinct points are contained in {\it exactly} one block is said to be a {\it linear space}.
In particular, a linear space with $v$ points in which every block has size $k$ is a {\it Steiner $2$-design} S$(2,k,v)$.

The {\it degree} of a point of a PLS is the number of blocks containing that point. A PLS has degree $r$ if all its points have the same degree $r$.
A PLS with $v$ points, constant block size $k$ and degree $r$ has necessarily $b={vr\over k}$ blocks and it is said to be a {\it configuration}.
It is often referred to as a $(v_r,b_k)$ {\it configuration} but here we prefer to write $(v,k;r)$ rather than $(v_r,b_k)$ in order to be consistent with
some other notation that will be used later. A configuration is said to be {\it symmetric} when the block size $k$ is equal to the degree $r$
(and consequently the numbers of points and blocks are equal).

A  {\it parallel class} of a PLS is a set of blocks partitioning the point set.
A PLS is said to be  {\it resolvable} if there exists a partition of the block set ({\it resolution}) into parallel classes.
Clearly, every resolvable PLS has degree equal to the number of its parallel classes.
By a {\it resolved} PLS we mean a resolvable PLS together with a {\it specific} resolution of it.
In particular, given an integer $v$ and a multiset of integers $\{k_1,\dots,k_r\}$, we will write $(v,\{k_1,\dots,k_r\})$-RPLS to denote a
resolved PLS whose resolution consists of $r$ parallel classes ${\mathcal P}_1, \dots, {\mathcal P}_r$ where all the blocks of $\mathcal P_i$ have size $k_i$
for $1\leq i\leq r$.

A resolvable configuration where the number of points is the square of the block size is said to be a {\it net}.
A net with block size $k$ and degree $r$ is often referred to as a $(r,k)$-{\it net} but here rather than to adopt this notation we prefer to speak of a $(k^2,k;r)$ net.

There is an extensive literature on resolvable linear spaces, especially on resolvable Steiner 2-designs. We point out, in particular,
that the existence of a resolvable S$(2,k,v)$ was proved in \cite{RW}
for $v$ admissible and sufficiently large.
For some literature on resolvable configurations we refer to \cite{BS21,BS22,G,S}.


\medskip
Let us return to ``the world of Heffter".
It is quite evident that a H$(m,n;h,k)$ gives rise to a resolvable PLS of degree 2 whose parallel classes are the
set of rows and the set of columns, respectively. This observation suggests to introduce the following new concept.
\begin{defn}\label{HS}
A {\it Heffter space} over an abelian group $G$
is a resolved PLS whose parallel classes are Heffter systems on a half-set of $G$.
Speaking of a $(v,\{k_1,\dots,k_r\})$ Heffter space we will mean a Heffter space with $r$ parallel classes
${\mathcal P}_1$, \dots, ${\mathcal P}_r$ where ${\mathcal P}_i$ is a $(v,k_i)$ Heffter system
for $1\leq i\leq r$. If $k_i=k$ for each $i$, it is clearly a $(v,k;r)$ configuration so that
we will speak of a $(v,k;r)$ {\it Heffter configuration}.
In the more special case  that $k_i=k$ for each $i$ and $v=k^2$, it is a net so that we
will speak of a $(k^2,k;r)$ {\it Heffter net}.
A Heffter space over $G$ will be said {\it cyclic} or {\it elementary abelian} if the group $G$ has the respective property.
\end{defn}

Note that a $(v,\{k_1,\dots,k_r\})$ Heffter space is equivalent to a $(v,\{k_1,\dots,k_r\})$-RPLS
whose points can be bijectively labeled with the elements of a half-set of an abelian group
in such a way that the sum of the labels in each block is zero. Such a labeling will be called a {\it Heffter labeling}.

According to Definition \ref{HS}, a $(v,k)$ Heffter system is nothing but a $(v,k;1)$ Heffter space or,  if you want,
a $(v,k;1)$ Heffter configuration. Also, a H$(m,n;h,k)$ is essentially the same as a $(nk,\{h,k\})$ Heffter space.
In particular, a H$(n;k)$ is essentially the same as a $(kn,k;2)$ Heffter configuration.

The aim of this paper is to find sets of mutually orthogonal Heffter systems as large as possible.
Heffter spaces are the crucial tool to investigate this problem in view of the following.

\begin{prop}\label{MOHS=HS}
A $(v,\{k_1,\dots,k_r\})$-MOHS is equivalent to a $(v,\{k_1,\dots,k_r\})$ Heffter space.
Thus, in particular, a $(v,k;r)$-MOHS is equivalent to a $(v,k;r)$ Heffter configuration.
\end{prop}
\begin{proof}
Let ${\mathcal R}=\{{\mathcal P}_1$, \dots, ${\mathcal P}_r\}$ be a $(v,\{k_1,\dots,k_r\})$-MOHS. Thus there is a half-set $V$ of an abelian group of order $2v+1$
such that ${\mathcal P}_i$ is a $(v,k_i)$ Heffter system on $V$ for $1\leq i\leq r$.
Consider the triple ${\mathcal S}=(V,{\mathcal B},{\mathcal R})$ where ${\mathcal B}=\bigcup_{i=1}^r{\mathcal P}_i$.
Given two distinct members $B$, $B'$ of ${\mathcal B}$, we have $B\in {\mathcal P}_i$ and $B'\in {\mathcal P}_j$ for suitable indices $i$ and $j$.
For $i=j$, $B$ and $B'$ are disjoint since they are distinct blocks of the same Heffter system whereas for $i\neq j$ we have $|B\cap B'|\leq1$ since
${\mathcal P}_i$ and ${\mathcal P}_j$ are orthogonal. This means that ${\mathcal S}$ is a PLS.
Also, by the definition of a Heffter system, each ${\mathcal P}_i$ is a partition of $V$ so that ${\mathcal R}$ is a resolution of ${\mathcal S}$.
We conclude that ${\mathcal S}$ is a $(v,\{k_1,\dots,k_r\})$ Heffter space.

Conversely, let ${\mathcal S}=(V,{\mathcal B},{\mathcal R})$ be a $(v,\{k_1,\dots,k_r\})$ Heffter space with
${\mathcal R}=\{{\mathcal P}_1$, \dots, ${\mathcal P}_r\}$ so that,
by definition, $\mathcal P_i$ is a $(v,k_i)$ Heffter system for each $i$. Any two distinct blocks
of ${\mathcal S}$ have at most one common point since ${\mathcal S}$ is a PLS.
Hence, in particular, for $i\neq j$, any block of ${\mathcal P}_i$ shares at most
one point with each block of ${\mathcal P}_j$, i.e., ${\mathcal P}_i$ and ${\mathcal P}_j$ are orthogonal. Thus ${\mathcal R}$ is a $(v,\{k_1,\dots,k_r\})$-MOHS.
\end{proof}

Let us determine a trivial necessary condition for the existence of a $(v,\{k_1,\dots,k_r\})$-MOHS.
\begin{prop}\label{collinear}
If there exists a $(v,\{k_1,\dots,k_r\})$-MOHS, then $k_1+{\dots}+k_r-r\leq v-1$ and equality holds if and only if
the associated Heffter space is linear.
\end{prop}
\begin{proof}
Let ${\mathcal R}=\{{\mathcal P}_1$, \dots, ${\mathcal P}_r\}$ be a $(v,\{k_1,\dots,k_r\})$-MOHS and let
${\mathcal S}=(V,{\mathcal B},{\mathcal R})$ be its associated Heffter space.
Take any point $x\in V$ and, for $1\leq i\leq r$, let $B_i$ be the only block of ${\mathcal P}_i$ containing $x$.
The set of points which are collinear with $x$ is $\bigcup_{i=1}^r(B_i\setminus\{x\})$ which has size $\sum_{i=1}^r(k_i-1)=k_1+{\dots}+k_r-r$.
The first part of the statement follows since, obviously, the number of points collinear with $x$ is at most $v-1$.
The second part also follows since to say that $\mathcal S$ is linear is equivalent to
say that every point $x\in V$ is collinear with any other point of $V$, i.e., that the number of points collinear with $x$ is
precisely $v-1$.
\end{proof}
Applying the above in the case that $k_i=k$ for each $i$ we get the following.
\begin{cor}\label{collinear2}
If there exists a $(v,k;r)$-MOHS, then $r(k-1)\leq v-1$ and the equality holds if and only if
the associated Heffter space is a S$(2,k,v)$.
\end{cor}
This immediately gives an upper bound on the maximum number of mutually orthogonal $(v,k)$ Heffter systems.
\begin{cor}\label{N(v,k)}
The maximum number of mutually orthogonal $(v,k)$ Heffter systems cannot exceed $\bigl{\lfloor}{v-1\over k-1}\bigl{\rfloor}.$
\end{cor}
Reaching the upper bound $\bigl{\lfloor}{v-1\over k-1}\bigl{\rfloor}$ seems to us a really exceptional occurrence. For instance, when $v-1$ is
divisible by $k-1$ this bound is precisely ${v-1\over k-1}$ and by Corollary \ref{collinear2}
the Heffter space associated with a $(v,k;{v-1\over k-1})$-MOHS would be a Heffter S$(2,k,v)$.
Yet, we are tempted to conjecture that a Heffter {\it linear}  space cannot exist. Indeed it should satisfy some necessary conditions that
we will establish in the next section and that appear to us extremely demanding.

Thus it is natural to say that a $(v,\{k_1,\dots,k_r\})$ Heffter space ${\mathcal S}=(V,{\mathcal B},{\mathcal R})$
is more interesting the closer it is to a linear space.
A good parameter to measure this ``distance" is the density ${|E|\over {|V|\choose2}}$ of the collinear graph of
$\mathcal S$, that is the graph $Col({\mathcal S})=(V,E)$
with vertex set $V$ where two vertices are adjacent if and only if they are collinear.
By the proof of Proposition \ref{collinear}, $Col({\mathcal S})$ is regular of degree $k_1+...+k_r-r$,
hence its density is equal to  ${k_1+...+k_r-r\over v-1}$.
It is maximum and equal to 1 precisely when $Col({\mathcal S})$ is complete and ${\mathcal S}$ is linear.
In case of a small density $Col({\mathcal S})$ is {\it sparse} and ${\mathcal S}$ is far from being linear.

Speaking of the density of a Heffter space $\mathcal S$ or its correspondent set of MOHSs we will mean
the density of the collinear graph $Col(\mathcal S)$.

\medskip
The construction for Heffter spaces of degree greater than 2 appears to us difficult even without the request of a
high density. In this paper we present some constructions for Heffter spaces with odd block sizes and an arbitrarily large degree
with the crucial help of finite fields. In most of the constructions we take advantage of the elementary fact that the set $\F_q^\Box$
of non-zero-squares of a field $\F_q$ of order $q\equiv3$ (mod 4) is a half-set of its additive group which, with an abuse of notation, will be also denoted by $\F_q$.

For now, getting Heffter spaces with even block sizes appears to us much more difficult. Some constructions will be presented
in a future paper \cite{BP3} where we will use a completely different approach which is not based on finite fields.

Our main results on MOHSs derive from the direct constructions for Heffter spaces obtained throughout the paper.
They can be summarized as follows.

A construction for Heffter spaces obtained in Section 3 using the notion of a {\it partial partition of a group} allows us to state
the following.
\begin{thm}\label{k1...krMOHS}
 If $k_1, \dots, k_r$ are mutually coprime
odd integers such that $2k_1\cdot{\dots}\cdot k_r+1$ is a prime power $q$, then there exists a $({q-1\over2},\{k_1,\dots,k_r\})$-MOHS.
\end{thm}

The above gives, for any integer $r$, infinite values of $v$ for which there exists a set of $r$ mutually orthogonal Heffter systems
of order $v$. On the other hand the block sizes of these systems are pairwise distinct whereas, especially for the applications,
it would be better to have the same result with a constant block size. This result, which is one of the strongest in this paper,
is a consequence of a construction obtained in Section 4 using {\it difference packings} and cyclotomy.

\begin{thm}\label{>8k^5r/k}
There exists a $({q-1\over2},k;r)$-MOHS for every triple $(q,k,r)$ such that $k\geq3$ is odd and $q\equiv2k+1$ $($mod $4k)$ is a prime power greater
than $8k^5\lceil{r\over k}\rceil$.
\end{thm}

We might be quite satisfied of this result but the density of these sets of MOHSs is very small since it is less than ${1\over 2k^4}$.
On the other hand, some computer results indicate that our constructions succeed for values of $q$ much less than the bound
obtained in the theorem and the density of the related sets of MOHSs is often greater than ${1\over2}$.

It is clear (also from Theorem \ref{thm:Heffter}) that $k^2$ is the smallest admissible $v$ for which it is possible to have a $(v,k;r)$ Heffter configuration
with $r\geq2$. In the extremal case of $v=k^2$ we have a Heffter net.
In Section 5 we will give a method to construct some $(k^2,k;4)$ Heffter nets when $k$ is odd and $2k^2+1$ is a prime power.
Using this method we have established that there exists a cyclic $(k^2,k;4)$ Heffter net for $k=9$, 21 and 27 with the help of a computer.
It is reasonable to believe that this method succeeds for any admissible $k>3$.

We have also constructed a $(121,11;9)$ Heffter net over $\F_{3^5}$ which deserves a special attention considering that it is the
densest Heffter configuration that we have found (its density is equal to ${3\over 4}$). Thus we have the following.

\begin{thm}\label{MOHSnet}
There exists a $(k^2,k;4)$-MOHS for $k=9, 21, 27$ and a $(121,11;9)$-MOHS.
\end{thm}


The previous results are exploited in Section 6 to construct sets of {\it mutually orthogonal cycle systems}.
First, we put a little bit of order in the literature concerning {\it orthogonal Steiner triple systems}.
Indeed, while according to some authors two S$(2,3,v)$ are orthogonal simply when they do not have any block in common,
according to many others (see, e.g., \cite{CGMMR,DDL,G,Gross,MN,shaug} and chapter 14.4 in \cite{CR}) they should satisfy an additional demanding condition that
we will reformulate in terms of near 1-factorizations of the complete graph. In order to avoid confusion, in the second case we will speak of
{\it super-orthogonal} Steiner triple systems.

Then we survey the main results on the existence of super-orthogonal Steiner triples systems and
on orthogonal $k$-cycle systems for any $k$.

Finally, as major application of our research on the Heffter spaces,
 we prove the following.
\begin{thm}\label{final}
If $q=2kw+1$ is a prime power with $kw$ odd and $k\geq 3$, then
there are at least $\lceil{w\over4k^4}\rceil$ mutually orthogonal $k$-cycle systems of order $q$.
\end{thm}


%
%

\section{Heffter linear spaces: do they exist?}

It is natural to ask whether there exists a Heffter space which is linear, in particular a Heffter S$(2,k,v)$.

To understand how hard this question is, it is convenient to make a small digression on the theory
of {\it additive designs} founded in \cite{CFP} and then developed in various papers such as \cite{BN1,BN2,CFP2,FP}.

A $t$-$(v,k,\lambda)$ design ${\mathcal D}=(V,{\mathcal B})$, in particular a S$(2,k,v)$, is said to be $G$-additive
if there exists a {\it $G$-additive labeling} of $\mathcal D$, i.e., an injective map
$f$ from $V$  to an abelian group $G$ such that $\sum_{x\in B}f(x)=0$
for every block $B\in{\mathcal B}$.

The definition can be extended, in the obvious way, to any combinatorial design.
For instance, in \cite{BM} {\it additive graph decompositions} will be considered.

Now note that the definition of a Heffter space can be equivalently formulated saying that it is
a $G$-additive $(v,\{k_1,\dots,k_r\})$-RPLS satisfying the ``half-set condition" that
the labels fill a half-set of $G$.

The last request becomes very demanding if we want a linear space.
Indeed we have to consider that the construction of additive Steiner 2-designs appears
to be very difficult,  even without any other additional request (see \cite{BN1,BN2}).
A Heffter Steiner 2-design besides being additive has to be resolvable
and has to satisfy the half-set condition.

In the next proposition we give a necessary condition showing, as a corollary, that the answer to our question is certainly
negative for cyclic Heffter spaces.

\begin{prop}\label{linear}
If there exists a Heffter linear space of degree $r$ over a group $G$,
then $r-1$ is a multiple of the order of any element $g\in G$.
\end{prop}
\begin{proof}
Let ${\mathcal S}$ be a Heffter linear space as in the statement and let $V$ be its point set so that we have $G = \{0\} \ \cup \ V \ \cup \ -V$.
It is enough to prove that $o(g)$ divides $r-1$ only for $g\in V$. Indeed the assertion is obvious for $g=0$ and the orders of the elements of $-V$ coincide with
the orders of their opposite which are in $V$.
So take $g\in V$ and let $B_1$, \dots, $B_r$ be the blocks of ${\mathcal S}$ passing through $g$.
Given that each $B_i$ is zero-sum, we have that the elements of $B_i\setminus\{g\}$
sum up to $-g$ for $1\leq i\leq r$. The fact that ${\mathcal S}$ is linear implies that $\{B_i\setminus\{g\} \ : \ 1\leq i\leq r\}$ is a partition of $V\setminus\{g\}$
so that the sum of all the elements of $V$ is equal to $r(-g)+g=(1-r)g$. On the other hand, as observed shortly after
Definition \ref{HeffterSystem}, the half-set $V$ is zero-sum. Thus the identity $(r-1)g=0$ holds in $G$, i.e., $o(g)$ divides $r-1$.
\end{proof}

\begin{cor}\label{linear2}
A cyclic Heffter space cannot be linear.
\end{cor}
\begin{proof}
Let $r$ be the degree of a Heffter linear space over $\Z_{2v+1}$.
By Proposition \ref{linear}, $r-1$ should be a multiple of the order of 1 that is $2v+1$. Thus the degree of a point would
be greater than the number $v$ of the points. This is obviously absurd.
\end{proof}

\begin{cor}\label{linear3}
A necessary condition for the existence of a Heffter linear space of order $v$ and degree $r$ is that every prime
divisor of $2v+1$ is also a divisor of $r-1$.
\end{cor}
\begin{proof}
Assume that there exists a Heffter linear space of order $v$ and degree $r$ over a group $G$
and let $p$ be a prime factor of $2v+1$.
By the theorem of Cauchy, the group $G$ has an element of order $p$ and then, by Proposition \ref{linear},
$p$ is a divisor of $r-1$.
\end{proof}

It is clear from the previous corollaries that the condition of Proposition \ref{linear} is very strict. Yet, we are going to see that there
are values of $v$ and $k$ for which this condition does not forbid a Heffter S$(2,k,v)$.
In the next proposition, given a natural number $n$,
we denote by $R_n$ the ring which is the direct product of all the fields whose orders are the
maximal prime power factors of $n$ so that, for instance, $R_{45}$ means $\F_9\times\F_5$.
Also, we denote by $rad(n)$ the {\it radical} of $n$, that is the product of all distinct prime divisors of $n$.
\begin{prop}\label{linear2}
Let $u\geq7$ be an odd integer with $u\not\equiv3$ $($mod $9)$.
 Then Proposition \ref{linear} does not exclude the existence of a Heffter S$(2,k,v)$ over $R_{u^n}$
 with $k={u-1\over2}$, $v={u^{n}-1\over2}$ and $n\equiv1$ $($mod $\omega)$ where $\omega$ is the multiplicative order of $3$ $($mod $u-3)$
 or $($mod ${u-3\over3})$ according to whether $3$ does not divide or divides $u-3$, respectively.
\end{prop}
\begin{proof}
Let $v$ and $k$ be two integers as in the statement. It is obvious that $v\equiv0$ (mod $k$). Also, by definition of $\omega$, it is easy to check that
$v-1\equiv0$ (mod $k-1$).
Thus the trivial necessary conditions for the existence of a resolvable S$(2,k,v)$ are satisfied.
The degree of such a putative S$(2,k,v)$ is the so-called {\it replication number} $r={v-1\over k-1}$ so that we have
$r-1=u\cdot {u^{n-1}-1\over u-3}$. It follows that $u$ divides $r-1$ since one can see that ${u^{n-1}-1\over u-3}$
is an integer by the definition of $\omega$ again. Also note that $2v+1=u^n$ and that the orders of the elements
of $R_{u^n}$ are the divisors of rad$(u)$. Indeed, if $u=p_1^{\alpha_1}\cdot{\dots}\cdot p_t^{\alpha_t}$ is
the prime power factorization of $u$, we have rad$(u)=p_1\cdot {\dots} \cdot p_t$ and
$R_{u^n}=\F_{p_1^{(\alpha_1n)}}\times{\dots}\times\F_{p_t^{(\alpha_tn)}}$. The order of an
element $x=(x_1,\dots,x_t)\in R_{u^n}$ is clearly equal to the product of all the primes $p_i$ with $x_i\neq0$.
Thus the order of any $x\in R_{u^n}$ is a divisor of $rad(u)$. It follows that it divides  $u$ and then $r-1$.
\end{proof}

The previous proposition applied with $u=7$ does not exclude that there exists a Heffter S$(2,3,{7^n-1\over2})$ with $n$ odd.
On the other hand, according to Theorem 3.7 in \cite{CFP}, the only additive S$(2,3,v)$ are the point-line designs associated with a projective geometry
over $\F_2$ and the point-line designs associated with an affine geometry over $\F_3$. In the first case $v$ is a Mersenne number,
in the second case $v$ is a power of 3. Thus, for a putative Heffter S$(2,3,{7^n-1\over2})$ we should have ${7^n-1\over2}=2^m-1$ or
${7^n-1\over2}=3^m$ for some $m$. This would mean that either $7^n+1=2^{m+1}$ or ${7^{n}-1\over7-1}=3^{m-1}$. But the first identity contradicts
the well-celebrated proof of Catalan's conjecture in \cite{M} that 8 and 9 are the only two consecutive perfect powers of natural numbers.
The second identity is also not possible in view of the known results on the famous {\it Nagell-Ljunggren equation} ${x^n-1\over x-1}=y^m$
(see \cite{BL}).

Applying Proposition \ref{linear2} with $u=9$ we deduce the possible existence of an S$(2,4,{9^n-1\over2})$ over $\F_{9^n}$.
The point-line design associated with the projective geometry PG$(2n-1,3)$ is actually a resolvable Steiner 2-design with
these parameters. It is also $\F_{9^n}$-additive in view of Theorem 5.1 in \cite{BN2}. But unfortunately, the
$\F_{9^n}$-additive labeling produced by the proof of this theorem is not a Heffter labeling since it does not satisfy the half-set condition.
Indeed the labels fill the set of squares of $\F_{9^n}$ which is not a half-set of $\F_{9^n}$.

Applying Proposition \ref{linear2} with $u=11$ we deduce the possible existence of an S$(2,5,{11^n-1\over2})$ over $\F_{11^n}$
with $n$ odd. For instance, we cannot exclude that there exists a Heffter S$(2,5,665)$ over $\F_{11^3}$.

\section{Heffter spaces via partial partitions of $\F_q^\Box$}

A set $\mathcal S$ of subgroups of a group $G$ covering the whole $G$ and intersecting each other trivially is said to be {\it a partition of $G$}
and gives rise to a resolvable linear space with special properties \cite{Andre}. In particular, when all the subgroups
have the same order $s$ and $G$ has order $s^2$ it is called a {\it partial congruence partition} of $G$ and it gives rise to a
{\it translation net} (see, e.g., \cite{BaJu,{HJ}}).
Here we speak of a {\it partial partition of a group $G$} (without the term ``congruence") to mean a set $\mathcal S$ of subgroups of $G$
intersecting each other trivially but not necessarily covering the whole $G$.
The following result is probably folklore but we prove it anyway for convenience of the reader.
\begin{prop}\label{Andre}
Let $\mathcal S=\{S_1,\dots,S_r\}$ be a partial partition of a group $G$ of order $v$, let $k_i$ be the order of $S_i$, and let $\mathcal P_i$ be
the set of all right cosets of $S_i$ in $G$ for $1\leq i\leq r$. Then $\mathcal P_1$, ...,  $\mathcal P_r$ are the parallel classes
of a $(v,\{k_1,\dots,k_r\})$-RPLS.
\end{prop}
\begin{proof}
Let $\mathcal B$ be the union of all the $\mathcal P_i$'s.
If two distinct elements $g, h$ of $G$ are contained in a right coset of a subgroup $S_i$,
we necessarily have $gh^{-1}\in S_i$. On the other hand, by definition of a partial partition of $G$ the element $gh^{-1}$ belongs
to at most one subgroup of ${\mathcal S}$. We conclude that $\{g,h\}$ is contained in at most one member of $\mathcal B$, i.e.,
$(G,{\mathcal B})$ is a partial linear space.  Each ${\mathcal P}_i$ is obviously a partition of $G$ for each $i$
so that $\{\mathcal P_1, ...,  \mathcal P_r\}$ is a resolution of $(G,{\mathcal B})$ and the assertion follows.
\end{proof}

\begin{thm}\label{planica}
Let $q=2v+1$ be a prime power with $v$ the product of $r$ mutually coprime odd integers $k_1,\dots,k_r$.
Then there exists a $(v,\{k_1,\dots,k_r\})$ Heffter space over $\F_q$.
\end{thm}
\begin{proof}
It is clear that $q\equiv3$ (mod 4) so that $\F_q^\Box$ is a half-set of $\F_q$.
For $1\leq i\leq r$, let $S_i$ be the group of $k_i$-th roots of unity of $\F_q$.
Given that the $k_i$'s are mutually coprime, the product $S_1\cdot {\dots} \cdot S_r$ is the subgroup of $\F_q^*$ of order $k_1\cdot {\dots}\cdot k_r={q-1\over2}$,
hence $S_1\cdot {\dots} \cdot S_r=\F_q^\Box$. Also, $\{S_1,\dots,S_r\}$ is clearly a partial partition of $\F_q^\Box$.
It follows, by Proposition \ref{Andre}, that if ${\mathcal P}_i$ is the set of cosets of $S_i$ in $\F_q^\Box$, then
$\mathcal P_1$, ...,  $\mathcal P_r$ are the parallel classes of a $(v,\{k_1,\dots,k_r\})$-RPLS with point set $\F_q^\Box$.
Now note that each block of this RPLS is zero-sum since any coset of a non-trivial multiplicative subgroup of a finite field
is zero-sum (see, e.g., Fact 1 in \cite{BN1}). It follows that every block of $\mathcal P_i$ is zero-sum, hence
$\mathcal P_i$ is a $(v,k_i)$ Heffter system on $\F_q^\Box$.
The assertion follows.
\end{proof}

Now it is clear that Theorem \ref{k1...krMOHS} is an immediate consequence of Theorem \ref{planica} and Proposition \ref{MOHS=HS}.

As an example, we can apply Theorem \ref{planica} with $r=3$ and $(k_1,k_2,k_3)=(3,5,7)$ since $2k_1k_2k_3+1=211$ is a prime.
We obtain in this way a $(105,\{3,5,7\})$ Heffter space.

Note that  Corollary 3.2 in \cite{B} is essentially Theorem \ref{planica} applied in the case $r=2$.

Let us show that Theorem \ref{planica} allows us to obtain Heffter spaces of arbitrarily large order.

\begin{thm}
There exists a Heffter space of degree $r$ for any positive integer $r$.
\end{thm}
\begin{proof}
In view of Theorem \ref{planica},
it is enough to show that for any given $r$ there exists an $r$-tuple $(k_1,\dots,k_r)$ of mutually coprime odd integers such that $2k_1\cdot{\dots}\cdot k_r+1$ is a prime.
Take $r-1$ mutually coprime odd integers $k_1$, \dots, $k_{r-1}$ (for instance the first $r-1$ odd primes) and set $m=2k_1\cdot {\dots}\cdot k_{r-1}$.
It is obvious that we have $\gcd(m^2,m+1)=1$ and hence, by the theorem of Dirichlet, the arithmetic progression
$\{m^2n+(m+1) \ | \ n\in\mathbb{N}\}$ has infinitely many primes.  Take one of these primes $p$ so that we can write $p=m^2n+(m+1)=m(mn+1)+1$
for a suitable $n$, and set $k_r=mn+1$. Note that $k_r$ is coprime with $m$ and hence coprime with each $k_i$
with $i\in\{1,\dots,r-1\}$. Also note that $k_r$ is odd since $m$ is even and that we have
$p=2k_1\cdot{\dots}\cdot k_{r-1}k_r+1$.
\end{proof}

A stronger version of the above result will be obtained in the next section by using {\it difference methods}.

\section{Heffter spaces via difference packings}

Given a subset $B$ of an additive (resp. multiplicative) group $G$, the {\it list of differences} of $B$ is the
multiset $\Delta B$ consisting of all possible differences $x-y$ (resp. ratios $xy^{-1}$) with $(x,y)$ an ordered pair of distinct elements of $B$.
More generally, one defines the list of differences of a collection ${\mathcal F}=\{B_1,\dots,B_n\}$ of subsets of $G$ as the multiset
$\Delta {\mathcal F}:=\Delta B_1 \ \cup \ \dots \ \cup \ \Delta B_n$.
The collection $\mathcal F$ is said to be a {\it difference packing} in $G$ if $\Delta {\mathcal F}$ does not have repeated elements.
It is a {\it difference family of index $1$} if $\Delta\mathcal F$ is precisely the set of all non-identity elements of $G$ (see e.g. \cite{BJL,CD}).
Speaking of a $(G,\{k_1,\dots,k_n\})$ difference packing we mean a difference packing $\{B_1,\dots,B_n\}$ in $G$ with $|B_i|= k_i$ for $1\leq i\leq n$.
We will write $(G,k;n)$ instead of $(G,\{k_1,\dots,k_n\})$ when $k_i=k$ for $1\leq i\leq n$.
Note, in particular, that a $(\Z_v,k;n)$ difference packing can also be seen as a $(v,k,1)$ {\it optical orthogonal code}.
The unique member of a $(\Z_v,k;1)$ difference packing is said to be a $(v,k)$ {\it modular Golomb ruler} (MGR).
A $(v,k)$-MGR is said to be {\it resolvable} if $k$ is a divisor of $v$ and its elements are pairwise distinct modulo $k$, i.e., they
form a complete set of residues (mod $k$) (see \cite{BS22}).

Now recall that the {\it development} of a subset $B$ of a group $G$ is the collection $dev B$ consisting of all possible translates of $B$ in $G$, that is
$dev B=\{B\star g \ | \ g\in G\}$ where $\star$ is the additive or multiplicative operation of $G$.
More generally, the development of a collection ${\mathcal F}=\{B_1,\dots,B_n\}$ of subsets of $G$ is the multiset union $dev B_1 \ \cup \ \dots \ \cup \ dev B_n$.
It is very well known that the pair $(G,dev{\mathcal F})$ is a partial linear space (resp. a linear space) if and only if ${\mathcal F}$ is a difference packing
(resp. a difference family of index $1$).
In particular, if $B$ is a $(v,k)$-MGR, then $(\Z_v,dev B)$ is a symmetric $(v,k;k)$-configuration which is resolvable if and only if $B$ is resolvable as well (see Theorem 4.3 in \cite{BS22}).

Considering that our aim is to construct Heffter spaces, we are interested in resolvable partial linear spaces generated by a difference packing.
A complete classification of them is beyond the scope of this paper.
We limit ourselves to consider a special class of them.

\begin{defn}\label{HDP}
Let $q=2v+1$ be a prime power with $v$ odd, let $\mathcal F=\{B_1,\dots,B_n\}$ be a set of subsets of $\F_q^\Box$, and let $\phi$ be an isomorphism
between $\F_q^\Box$ and $\Z_{v}$. We say that $\mathcal F$ is a {\it $(\F_q^\Box,\{k_1,\dots,k_n\})$ Heffter difference packing} if the following conditions hold:
\begin{itemize}
\item[(H$_1$)]\quad $\phi({\mathcal F})$ is a $({v},\{k_1,\dots,k_n\})$ difference packing;
\item[(H$_2$)]\quad $\phi(B_i)$ is a resolvable $(v,k_i)$-MGR for $1\leq i\leq n$;
\item[(H$_3$)]\quad each $B_i$ is zero-sum in $\F_q$.
\end{itemize}
The only member of a $(\F_q^\Box,\{k\})$ Heffter difference packing will be called a {\it $(\F_q^\Box,k)$ Heffter ruler}.
\end{defn}

\begin{rem}\label{Heffterruler}
It is convenient to observe that $\mathcal F$ is a $(\F_q^\Box,k;n)$ Heffter difference packing if and only if
it is a $n$-set of $(\F_q^\Box,k)$ Heffter rulers whose lists of differences are pairwise disjoint.
\end{rem}

The above terminology is justified by the following.

\begin{thm}\label{HDP->HS}
If $q=2v+1$ is a prime power with $v$ odd, and there exists a $(\F_q^\Box,\{k_1,\dots,k_n\})$ Heffter difference packing, then there exists
a $(v,\{k_1^{k_1},\dots,k_n^{k_n}\})$ Heffter space. In the additional hypothesis that the difference packing is simple,
the resulting Heffter space is simple too.
\end{thm}
\begin{proof}
Let ${\mathcal F}=\{B_1,\dots,B_n\}$ be a $(\F_q^\Box,\{k_1,\dots,k_n\})$ Heffter difference packing.

Given any pair $(x,y)$ of elements of $\F_q^\Box$ we have $\phi(xy^{-1})=\phi(x)-\phi(y)$ since $\phi$ is a group isomorphism.
Thus we have
$\phi(\Delta B)=\Delta\phi(B)$ for every subset $B$ of $\F_q^\Box$ and then we can write
$$\phi(\Delta B_1 \ \cup \ \dots \ \cup \ \Delta B_n)=\phi(\Delta B_1) \ \cup \ \dots \ \cup \ \phi(\Delta B_n)=\Delta\phi(B_1) \ \cup \ \dots \ \cup \ \Delta\phi(B_n).$$
Thus, given that $\phi$ is bijective and that condition (H$_1$) says that $\Delta\phi(B_1) \ \cup \ \dots \ \cup \ \Delta\phi(B_n)$
does not have repeated elements, we can say that $\Delta B_1 \ \cup \ \dots \ \cup \ \Delta B_n$ is also free of repetitions.
This means that $\mathcal F$ is a full-fledged $(\F_q^\Box,\{k_1,\dots,k_n\})$ difference packing so that
$(\F_q^\Box,dev{\mathcal F})$ is a PLS whose point set $\F_q^\Box$ is a half-set of $\F_q$ since $q\equiv3$ (mod 4)
by assumption.

For $1\leq i\leq n$, condition (H$_2$) says that $k_i$ divides ${v}$ and that the elements of $\phi(B_i)$ form a complete set of residues (mod $k_i$).
Equivalently, they form a complete system of representatives for the cosets of the subgroup of $\Z_{v}$ of index $k_i$, that is $k_i\Z_{v}$.
For $i=1,\dots,n$, let $S_i$ be the subgroup of index $k_i$ of $\F_q^\Box$, that is the group of non-zero $k_i$th powers of $\F_q$.
Given that $\phi$ is a group isomorphism, we have that $B_i$ is a complete system of representatives for the cosets of $S_i$
in $\F_q^\Box$ for $1\leq i\leq n$.
It follows that
${\mathcal P}_i:=\{B_is \ | \ s\in S_i\}$ is a partition of $\F_q^\Box$, i.e.,
a parallel class of $(\F_q^\Box,dev{\mathcal F})$.
Take a complete system $T_i$ of representatives for the cosets of $S_i$ in $\F_q^\Box$ so that we have $\{st \ | \ s\in S_i; t\in T_i\}=\F_q^\Box$.
We can take, for instance, $T_i=\{g^j \ | \ 0\leq j\leq k_i-1\}$ with $g$ a generator of $\F_q^\Box$.
Since ${\mathcal P}_i$ is a parallel class, ${\mathcal P}_it:=\{B_ist \ | \ s\in S_i\}$ is a parallel class as well for every $t\in T_i$.
Thus we see that ${\mathcal R}_i:=\{{\mathcal P}_it \ | \ t\in T_i\}$ is a partition of $dev B_i$ into $k_i$ parallel classes.
We conclude that ${\mathcal R}:=\bigcup_{i=1}^n{\mathcal R}_i$ is a resolution of $(\F_q^\Box,dev{\mathcal F})$,
hence $(\F_q^\Box,dev{\mathcal F},{\mathcal R})$ is a $(v,\{k_1^{k_1},\dots,k_n^{k_n}\})$-RPLS.
Finally, any block $B$ of $dev \mathcal F$ is of the form $xB_i$ for a suitable $x\in\F_q^\Box$ and a suitable $i\in\{1,\dots,n\}$.
The elements of $B$ sum up to $xs$ where $s$ is the sum of the elements of $B_i$.
Thus $B$ is zero-sum since $s=0$ by condition (H$_3$) and then every member of ${\mathcal R}$ is a Heffter system.

Finally note that if $\mathcal F$ is simple, then every member of $dev\mathcal F$ is simple as well. The assertion follows.
\end{proof}

\begin{cor}\label{HDP->HS2}
A $(\F_q^\Box,k;n)$ Heffter difference packing gives rise to a $({q-1\over 2},k;kn)$ Heffter configuration
which is simple if the difference packing is simple as well. Also, this configuration can be extended to a $({q-1\over2},\{k^{kn},{q-1\over2k}\})$ Heffter space.
\end{cor}
\begin{proof}
Let ${\mathcal F}=\{B_1,\dots,B_n\}$ be a $(\F_q^\Box,k;n)$ Heffter difference packing.
Applying Theorem \ref{HDP->HS} we immediately get a $(v,k;kn)$ Heffter space $\mathcal S=(\F_q^\Box,dev{\mathcal F},{\mathcal R})$
that is a $({q-1\over 2},k;kn)$ Heffter configuration which is simple if $\mathcal F$ is simple as well.
Recall that the $k$ elements of any member of $\mathcal F$ are in pairwise distinct cosets of the subgroup
$S$ of $\F_q^\Box$ of index $k$. It follows that the $k$ elements of any block of $dev \mathcal F$ also belong to pairwise distinct cosets of $S$.
This means that any two distinct elements of  $\F_q^\Box$ lying in the same coset of $S$
are never collinear in $\mathcal S$. Thus, if ${\mathcal C}$ is the set of cosets of $S$ in $\F_q^\Box$, the pair $(\F_q^\Box,dev{\mathcal F} \ \cup \ {\mathcal C})$
is still a partial linear space. Also note that the blocks of ${\mathcal C}$ have size ${q-1\over2k}$ and form a parallel class of this space.
We conclude that the triple $(\F_q^\Box,dev{\mathcal F} \ \cup \ {\mathcal C},{\mathcal R} \ \cup \ \{{\mathcal C}\})$ is a $({q-1\over2},\{k^{kn},{q-1\over2k}\})$ Heffter space.
\end{proof}

\begin{ex}\label{35} Let us construct a symmetric $(35,5;5)$ Heffter configuration which is the same as a $(35,5;5)$-MOHS. For this, we need a $(\F_{71}^\Box,5)$ Heffter ruler.
First note that $3$ is a generator of $\Z_{71}^\Box$ so that the map
$\phi: 3^i \in \Z_{71}^\Box \longrightarrow i\in \Z_{35}$ is a group isomorphism.
Consider the 5-subset $B=\{3^0,3^{21},3^{27},3^{18},3^{34}\}$ of $\Z_{71}^\Box$ for which we obviously have
$\phi(B)=\{0,21,27,18,34\}$. Looking at the {\it difference table} of $\phi(B)$ below
\begin{center}
\begin{tabular}{|c|||c|c|c|c|c|c|c|c|c|c|c|c|c|c}
\hline {$$} & {\scriptsize$0$} & {\scriptsize$21$} & {\scriptsize$27$} & {\scriptsize$18$} & {\scriptsize$34$}   \\
\hline\hline\hline {\scriptsize$0$} & $\bullet$ & $14$ & $8$ & $17$ & $1$  \\
\hline {\scriptsize$21$} & $21$ & $\bullet$ & $29$ & $3$ & $22$  \\
\hline {\scriptsize$27$} & $27$ & $6$ & $\bullet$ & $9$ & $28$  \\
\hline {\scriptsize$18$} & $18$ & $32$ & $26$ & $\bullet$ & $19$  \\
\hline {\scriptsize$34$} & $34$ & $13$ & $7$ & $16$ & $\bullet$  \\
\hline
\end{tabular}\quad\quad\quad
\end{center}
we see that $\Delta\phi(B)$ does not have repeated elements, hence $\phi(B)$ is a $(35,5)$-MGR.

It is readily seen that the elements of $\phi(B)$ are pairwise distinct modulo 5 so that $\phi(B)$ is resolvable.
Finally, noting that $B=\{1, 25, 49, 43, 24\}$, we see that $B$ is zero-sum in $\Z_{71}$.
We conclude that $B$ is a $(\Z_{71}^\Box,5)$ Heffter ruler.

Following the instructions given in the proof of Theorem \ref{HDP->HS} we find a symmetric $(35,5;5)$ Heffter configuration
with point set $\Z_{71}^\Box$ and resolution ${\mathcal R}=\{\mathcal {P}_1$, \dots, $\mathcal {P}_5\}$ as follows:

\small
\medskip\noindent
\begin{center}
\begin{tabular}{|l|c|r|c|r|c|r|c|r|c|r|c|r|}
\hline {\quad\quad\quad$\mathcal {P}_1$}    \\
\hline $\{1, 24, 25, 43, 49\}$\\
\hline $\{45, 15, 60, 18, 4\}$  \\
\hline $\{37, 36, 2, 29, 38\}$\\
\hline $\{32, 58, 19, 27, 6\}$  \\
\hline $\{20, 54, 3, 8, 57\}$\\
\hline $\{48, 16, 64, 5, 9\}$ \\
\hline $\{30, 10, 40, 12, 50\}$ \\
\hline
\end{tabular}\quad
\begin{tabular}{|l|c|r|c|r|c|r|c|r|c|r|c|r|}
\hline {\quad\quad\quad$\mathcal{P}_2$}    \\
\hline $\{49, 40, 18, 48, 58\}$  \\
\hline $\{4, 25, 29, 30, 54\}$\\
\hline $\{38, 60, 27, 1, 16\}$ \\
\hline $\{6, 2, 8, 45, 10\}$\\
\hline $\{57, 19, 5, 37, 24\}$  \\
\hline $\{9, 3, 12, 32, 15\}$\\
\hline $\{50, 64, 43, 20, 36\}$\\
\hline
\end{tabular}\quad
\begin{tabular}{|l|c|r|c|r|c|r|c|r|c|r|c|r|}
\hline {\quad\quad\quad$\mathcal{P}_3$}    \\
\hline $\{58, 43, 30, 9, 2\}$\\
\hline $\{54, 18, 1, 50, 19\}$\\
\hline $\{16, 29, 45, 49, 3\}$ \\
\hline $\{10, 27, 37, 4, 64\}$  \\
\hline $\{24, 8, 32, 38, 40\}$\\
\hline $\{15, 5, 20, 6, 25\}$\\
\hline $\{36, 12, 48, 57, 60\}$\\
\hline
\end{tabular}\quad
\begin{tabular}{|l|c|r|c|r|c|r|c|r|c|r|c|r|}
\hline {\quad\quad\quad$\mathcal{P}_4$}    \\
\hline $\{2, 48, 50, 15, 27\}$\\
\hline $\{19, 30, 49, 36, 8\}$\\
\hline $\{3, 1, 4, 58, 5\}$ \\
\hline $\{64, 45, 38, 54, 12\}$  \\
\hline $\{40, 37, 6, 16, 43\}$\\
\hline $\{25, 32, 57, 10, 18\}$\\
\hline $\{60, 20, 9, 24, 29\}$\\
\hline
\end{tabular}\quad
\begin{tabular}{|l|c|r|c|r|c|r|c|r|c|r|c|r|}
\hline {\quad\quad\quad$\mathcal{P}_5$}    \\
\hline $\{27, 9, 36, 25, 45\}$\\
\hline $\{8, 50, 58, 60, 37\}$\\
\hline $\{5, 49, 54, 2, 32\}$ \\
\hline $\{12, 4, 16, 19, 20\}$  \\
\hline $\{43, 38, 10, 3, 48\}$\\
\hline $\{18, 6, 24, 64, 30\}$\\
\hline $\{29, 57, 15, 40, 1\}$\\
\hline
\end{tabular}
\end{center}
\end{ex}

\medskip\normalsize
The $(35,5;5)$ Heffter configuration above has density ${10\over17}\simeq0.588$. It is the second most dense
Heffter configuration constructed in this paper. Adding the set
\begin{center}
\begin{tabular}{|l|c|r|c|r|c|r|c|r|c|r|c|r|}
\hline {\quad\quad\quad$\mathcal{C}$}    \\
\hline $\{1, 45, 37, 32, 20, 48, 30\}$\\
\hline $\{49, 4, 38, 6, 57, 9, 50\}$\\
\hline $\{58, 54, 16, 10, 24, 15, 36\}$ \\
\hline $\{2, 19, 3, 64, 40, 25, 60\}$  \\
\hline $\{27, 8, 5, 12, 43, 18, 29\}$\\
\hline
\end{tabular}
\end{center}

\medskip\noindent
of all the cosets of the subgroup of $\Z_{71}^\Box$ of order 7, we get a $(35,\{5^5,7\})$ Heffter space. This is the densest Heffter space that
we have at this moment; its density is ${13\over17}\simeq0.7647$.

\medskip
From now on, given a prime power $q\equiv1$ (mod $e$), the subgroup of $\F_q^*$ of
index $e$ will be denoted by $C^e$. Thus, for $q$ odd, $C^2$ is nothing but $\F_q^\Box$. Given a primitive element $\rho$ of $\F_q$, the set of cosets of $C^e$ in $\F_q^*$
is $\{\rho^iC^e \ | \ 0\leq i\leq e-1\}$. As it is standard, the coset $\rho^iC^e$ will be denoted by $C^e_i$.
Note that we have $C^e_i\cdot C^e_j=C^e_{i+j \ (mod \ e)}$. Also note that if $k$ is a divisor of ${q-1\over2}$, then the cosets of $C^{2k}$
in $\F_q^\Box$ are $C^{2k}_{2i}$ with $0\leq i\leq k-1$.

We will need the following lemma which is a special case of Theorem 1.1 in \cite{BP}.
\begin{lem}\label{BP}
Let $q\equiv1$ $($mod $e)$ be a prime power, let
$(i,j)\in \Z_e\times\Z_e$, and let
$s$ be a non-zero element of $\F_q$. Then the inequality
$$\bigl{|}\{x\in C^e_i \ : \ x-s\in C^e_j\}\bigl{|}>t$$
is guaranteed at least for $q>Q(e,t)$ where
$$Q(e,t)={1\over4}\biggl{[}(e-1)^2+\sqrt{(e-1)^4+4e(et+2)}\biggl{]}^2.$$
\end{lem}

The following result is very technical but crucial to prove our main result in this section, that is Corollary \ref{asymptotic3}.
Keep in mind the meaning of ``simple" explained in Definition \ref{simple}.

\begin{thm}\label{asymptotic}
Given any odd $k\geq3$ and any $n\geq0$, there exists a simple
$(\F_q^\Box,k;n)$ Heffter difference packing for every prime power  $q\equiv2k+1$ $($mod $4k)$ greater than $8k^5n$.
\end{thm}
\begin{proof}
First note that $q\equiv2k+1$ (mod $4k$) implies that $q\equiv1$ (mod $k$) and that $q\equiv3$ (mod 4) since $k$ is odd.

We proceed by induction on $n$. The assertion is obvious for $n=0$; the empty family can be considered as a $(\F_q^\Box,k;0)$ Heffter difference packing.
Assume the assertion true for $n-1$ and let $q\equiv 2k+1$ (mod $4k$) be a prime power with $q>8k^5n$ so that, obviously, $q>8k^5(n-1)$. Then, by induction,
there exists a simple $(\F_q^\Box,k;n-1)$ Heffter difference packing ${\mathcal F}$.
In view of Remark \ref{Heffterruler} it is enough to prove that
there exists a simple $(\F_q^\Box,k)$ Heffter ruler $B$ such that $\Delta B$ is disjoint with $\Delta{\mathcal F}$
so that ${\mathcal F} \ \cup \ \{B\}$ will be a $(\F_q^\Box,k;n)$ Heffter difference packing.

The required ruler $B$ will be constructed using the following strategy. We first construct a simple $(k-2)$-subset $B^*=\{b_0, b_1, \dots, b_{k-3}\}$ of $\F_q^\Box$
such that:
\begin{itemize}
\item[$(R_1)$]\quad $b_i\in C^{2k}_{2i}$ for $0\leq i\leq k-3$;

\smallskip
\item[$(R_2)$]\quad $\Delta B^* \ \cup \ \Delta{\mathcal F}$ does not have repeated elements;

\smallskip
\item[$(R_3)$]\quad no partial sum of $B^*$ is null.
\end{itemize}
After that, let us denote by $s$ the sum of all the elements of $B^*$ and then look for a field element $x$ such that:
\begin{itemize}
\item[$(R_4)$]\quad $x\in C^{2k}_{2(k-2)}$;

\medskip
\item[$(R_5)$]\quad $-x-s\in C^{2k}_{2(k-1)}$;

\medskip
\item[$(R_6)$]\quad the following inequalities hold:

\smallskip $\begin{array}{ccccccc}

\smallskip
(R_{6,1}) &&b_ix\neq-b_j(x+s) &&&& \hfill{\rm for} \ (i,j)\in\{0,1,\dots,k-3\}^2  \\

\smallskip
(R_{6,2}) &&{x(x+s)}\neq-b_ib_j &&&& \hfill{\rm for} \  0\leq i\leq j\leq k-3  \\

\smallskip
(R_{6,3}) &&x^2\neq b_ib_j &&&& \hfill{\rm for} \  0\leq i\leq j\leq k-3   \\

\smallskip
(R_{6,4}) &&(x+s)^2\neq b_ib_j &&&& \hfill{\rm for} \  0\leq i\leq j\leq k-3   \\

\smallskip
(R_{6,5}) &&x^2\neq -b_i(x+s) &&&& \hfill{\rm for} \  0\leq i\leq k-3   \\

\smallskip
(R_{6,6}) &&(x+s)^2\neq b_ix &&&& \hfill{\rm for} \  0\leq i\leq k-3   \\


\smallskip
(R_{6,7}) &&{x}\neq b_i\delta &&&& \hfill{\rm for} \ 0\leq i\leq k-3 \ \ {\rm and} \ \ \delta\in \Delta B^* \ \cup \ \Delta{\mathcal F} \\

\smallskip
(R_{6,8}) &&x+s\neq -b_i\delta &&&&\hfill{\rm for} \  0\leq i\leq k-3  \ \ {\rm and} \ \ \delta\in \Delta B^* \ \cup \ \Delta{\mathcal F} \\

\smallskip
(R_{6,9}) &&x+s\neq -\delta x &&&& \hfill{\rm for} \ \delta\in \Delta B^* \ \cup \ \Delta{\mathcal F} \\

\smallskip
(R_{6,10}) &&x+s\neq \sum_{j=0}^ib_j &&&& \hfill{\rm for} \  0\leq i\leq k-3.  \\
\end{array}$
\end{itemize}

Let us show that all these conditions guarantee that
\begin{equation}\label{goodB}
\mbox{$B:=B^* \ \cup \ \{x,-x-s\}$ is a simple $(\F_q^\Box,k)$ Heffter ruler with $\Delta B \ \cap \ \Delta{\mathcal F}=\emptyset$}
\end{equation}
as required.

We have $$\Delta B=\Delta B^* \ \cup \ \Delta'\quad {\rm with} \quad \Delta'=\biggl{\{}{x\over b_i},{b_i\over x},{-x-s\over b_i},{b_i\over -x-s} \ \biggl{|} \ 0\leq i\leq k-3\biggl{\}} \ \cup \ \biggl{\{}{x\over -x-s},{-x-s\over x}\biggl{\}}.$$
We know that $\Delta B^*$ does not have repeated elements because of $(R_2)$ and it is easily seen that
inequalities $(R_{6,1}), \dots, (R_{6,6})$ essentially say that
$\Delta'$ does not have repeated elements. Also, $(R_{6,7})$, $(R_{6,8})$, $(R_{6,9})$ assure that $\Delta B^*$ and $\Delta'$ are disjoint.
It follows that $\Delta B$ does not have repeated elements, hence $\phi(B)$ is a $({q-1\over 2},k)$-MGR
for any group isomorphism $\phi: \F_q^\Box \longrightarrow \Z_v$.

The cosets of $C^{2k}$ in $\F_q^\Box$ are $C^{2k}_{2i}$ with $i=0,1,\dots,k-1$.
Hence $B$ is a complete system of representatives for the cosets of $C^{2k}$ in $\F_q^\Box$ by conditions $(R_1)$, $(R_4)$, $(R_5)$.
This is equivalent to say that the elements of $\phi(B)$
form a complete set of residues modulo $k$. Thus $\phi(B)$ is resolvable.

Condition $(R_3)$ implies that $B$ is zero-sum and
inequality $(R_{6,10})$ guarantees that $B$ is simple.
Hence $B$ is a simple $(\F_q^\Box,k)$ Heffter ruler.

Finally, inequalities $(R_{6,7})$, $(R_{6,8})$, $(R_{6,9})$ also guarantee that $\Delta' \ \cap \ \Delta{\mathcal F}=\emptyset$ so that
$\Delta B\ \cap \ \Delta{\mathcal F}=\emptyset$ by $(R_2)$ and then (\ref{goodB}) follows.

It might seem that we also needed the inequalities $x(x+s)\neq-b_ix$ and $x(x+s)\neq b_i(x+s)$ with $0\leq i\leq k-3$,
and $(x+s)^2\neq x^2$, the latter in order to avoid that the last two elements of $\Delta'$ coincide.
As a matter of fact they should be redundant.
Indeed $x(x+s)\neq-b_ix$ is already contained in the set of inequalities $(R_{6,4})$.
Also, $x(x+s)\neq b_i(x+s)$ certainly holds otherwise we would have either $x=-s$ against $(R_5)$ or
$x=b_i$ against $(R_1),(R_4)$. Finally, $(x+s)^2=x^2$ would mean either $x=-x-s$ which is in conflict with $(R_4),(R_5)$, or $s=0$ against $(R_3)$.

And then all we have to do is prove that a simple set $B^*=\{b_0,b_1,\dots,b_{k-3}\}$ satisfying $(R_1)$,  $(R_2)$,  $(R_3)$
and an element $x$ satisfying $(R_4),(R_5),(R_6)$ actually exist.

The set $B^*$ can be determined as follows. Start taking $b_{0}$ in $C^{2k}_0$ arbitrarily and then take the other elements
$b_{1}$, $b_{2}$, \dots, $b_{k-3}$ iteratively, one by one, according to the rule that
once that $b_{j-1}$ has been chosen, we pick $b_{j}$ arbitrarily in the set
$$C^{2k}_{2j}\setminus(\Delta{\mathcal F} \ \cup \ \Delta_{j-1} \ \cup \ S_{j-1})$$
where $\Delta_{j-1}=\Delta\{b_{0},b_{1},\dots,b_{j-1}\}$ and $S_{j-1}$ is the set of all elements of the form $-\sum_{h=\alpha}^\beta b_h$
with $0\leq \alpha\leq\beta\leq j-1$.
This is possible since $C^{2k}_{2j}$ has size ${q-1\over2k}$ whereas
$\Delta{\mathcal F} \ \cup \ \Delta_{j-1} \ \cup \ S_{j-1}$ has size at most equal to $k(k-1)(n-1)+j(j-1)+{j(j-1)\over2}$ which is certainly less than ${q-1\over2k}$
in view of the hypothesis $q>8k^5n$.

Note that $(R_5)$ can be rewritten as $x+s\in C^{2k}_{k-2}$ since $-1\in C^{2k}_k$.
Hence, an element $x$ satisfies conditions $(R_4),(R_5),(R_6)$ if and only if it belongs to the set
$$X=\{x\in C^{2k}_{2k-4}\setminus F \ : \ x+s\in C^{2k}_{k-2}\}$$
where $F$ is the set of values forbidden by condition $(R_6)$.
So we have to prove that $X$ is not empty or, equivalently, that the set
$$X'=\{x\in C^{2k}_{2k-4} \ : \ x+s\in C^{2k}_{k-2}\}$$
has size greater than $|F \ \cap \ C^{2k}_{2k-4}|$.

Let us denote by $F_i$  the set of values forbidden by the set of inequalities $(R_{6,i})$
for $1\leq i\leq 10$,
and let us show that $|F_{i} \ \cap \ C^{2k}_{2k-4}|$ is bounded above by $u_i$ as indicated  in the following table.
\begin{center}
\begin{tabular}{|c||c|}
\hline {$i$} & $u_i$    \\
\hline
\hline {$1$} & $(k-2)^2$    \\
\hline {$2$} & $(k-2)(k-1)$    \\
\hline {$3$} & $(k-2)(k-1)$    \\
\hline {$4$} & $(k-2)(k-1)$    \\
\hline {$5$} & $2(k-2)$    \\
\hline {$6$} & $2(k-2)$    \\
\hline {$7$} & $k-2+(k-1)(n-1)$    \\
\hline {$8$} & $(k-2)[(k-2)(k-3)+k(k-1)(n-1)]$    \\
\hline {$9$} & $(k-2)(k-3)+k(k-1)(n-1)$    \\
\hline {$10$} & $k-2$    \\
\hline
\end{tabular}
\end{center}
For $i\neq 7$, we have got $|F_{i} \ \cap \ C^{2k}_{2k-4}|\leq u_i$ by simply counting the number of inequalities in $(R_{6,i})$ taking into account their degrees.
For instance,
$(R_{6,1})$ is a set of $(k-2)^2$ inequalities of degree 1. Hence $|F_1|$ and, a fortiori, $|F_1 \ \cap \ C^{2k}_{2k-4}|$ are at most equal to $(k-2)^2=u_1$.
Analogously, $(R_{6,2})$ is a set of ${(k-2)(k-1)\over2}$ inequalities of degree 2. Hence
$|F_2|$ and, a fortiori, $|F_2 \ \cap \ C^{2k}_{2k-4}|$ are at most equal to $2\cdot{(k-2)(k-1)\over2}=u_2$.

Now we explain how we have got $u_7$.
We have
$$F_7=\{b_i\delta \ | \ 0\leq i\leq k-3; \delta\in \Delta B^* \ \cup \ \Delta{\mathcal F}\}$$
but $F_7 \ \cap \ C^{2k}_{2k-4}$ is considerably smaller for the following reasons.
First observe that we have
\begin{equation}\label{u7}
b_i\delta \in C^{2k}_{2k-4} \Longleftrightarrow \delta\in C^{2k}_{2k-4-2i}
\end{equation}
since $b_i$ has been taken in $C^{2k}_{2i}$.
The fact that the elements of every member of $\mathcal F$ are in pairwise distinct cosets of
$C^{2k}$ in $\F_q^\Box$ implies that $\Delta{\mathcal F}$ is evenly distributed over these cosets.
It follows that $\Delta{\mathcal F}$ contains precisely ${|\Delta{\mathcal F}|\over k}=(k-1)(n-1)$ elements
in each of these cosets.
Analogously, the fact that the elements of $B^*$ are in pairwise distinct cosets of
$C^{2k}$ in $\F_q^\Box$ implies that $\Delta B^*$ has at most $k-2$ elements in each of these cosets.
Thus, given $i\in\{0,1,\dots,k-3\}$, the number of elements $\delta\in (\Delta B^* \ \cup \ \Delta{\mathcal F}) \cap C^{2k}_{2k-4-2i}$
is at most $k-2+(k-1)(n-1)$. We conclude that $|F_7 \ \cap \ C^{2k}_{2k-4}|\leq u_7$ in view of (\ref{u7}).

Thus we have
$$|F \ \cap \ C^{2k}_{2k-4}|\leq\sum_{i=1}^{10}|F_i \ \cap \ C^{2k}_{2k-4}|\leq \sum_{i=1}^{10}u_i=k^2(k-1)n+(1-2k)n-3$$
and then we can write
\begin{equation}\label{forbidden}
|F \ \cap \ C^{2k}_{2k-4}|<k^2(k-1)n.
\end{equation}
Now consider the function $Q(e,t)$ in Lemma \ref{BP} and calculate $Q(2k,k^2(k-1)n)$. We have
$$Q(2k,k^2(k-1)n)={1\over4}\biggl{[}(2k-1)^2+\sqrt{(2k-1)^4+8k(2k^4n-2k^3n+2)}\biggl{]}^2<8k^5n<q.$$
Thus, the same lemma applied with $e=2k$, $(i,j)=(2k-4,k-2)$, and $t=k^2(k-1)n$, assures that $|X'|>k^2(k-1)n$.
Then, comparing with (\ref{forbidden}), we conclude that $|X'|>|F \ \cap \ C^{2k}_{2k-4}|$. The assertion follows.
\end{proof}

\begin{cor}\label{asymptotic2}
Given any odd $k\geq3$ and any $r>0$, there exists a simple $(v,k;r)$ Heffter configuration for every odd integer $v=kw$ such that
$2v+1$ is a prime power and $w\geq4k^4\lceil{r\over k}\rceil$.
\end{cor}
\begin{proof}
Set $n=\lceil{r\over k}\rceil$ and let $v$ be an integer as in the statement so that $q=2v+1$ is a prime power congruent to $2k+1$ $($mod $4k)$
and greater than $8k^5n$. Thus there exists a simple $(\F_q^\Box,k;n)$ Heffter difference packing by Theorem \ref{asymptotic} and then
a simple $(v,k;kn)$ Heffter configuration by Corollary \ref{HDP->HS2}. Deleting $kn-r$ parallel classes
we get the desired $(v,k;r)$ Heffter configuration.
\end{proof}

Now note that Theorem \ref{>8k^5r/k} is an  immediate consequence of Corollary \ref{asymptotic2} and Proposition \ref{MOHS=HS}.
Also note that Corollary \ref{asymptotic2} together with Dirichlet's theorem on arithmetic progressions, allows us to state the following.

\begin{cor}\label{asymptotic3}
Given any odd $k\geq3$ and any $r>0$, there are infinitely many values of $v$ for which
there exists a simple $(v,k;r)$ Heffter configuration.
\end{cor}

Even though the bound $q>8k^5n$ of Theorem \ref{asymptotic} is not so huge as in other {\it Weil-type constructions}, our
computer results show that it is much greater than necessary anyway.
Suffice it to say that the first $q$ for which Theorem \ref{asymptotic} guarantees the existence
of a $(\F_q^\Box,5)$ Heffter ruler is $25,031$ whereas in Example \ref{35} we got such a
Heffter ruler with $q$ equal to just $71$. Further evidence to this fact is provided by the
following table reporting the minimum $q$ for which, given $k\leq13$, there exists a simple $(\F^\Box_q,k)$ Heffter ruler,
an example of such a ruler, the density of its related $({q-1\over2},k;k)$ Heffter symmetric configuration (first density), and the density
of its related $({q-1\over2},\{k^k,{q-1\over2k}\})$ Heffter space (second density).
\begin{center}
\begin{tabular}{|c|c|c|c|c|}
\hline {$k$} & $q_{min}$ & $(\F_q^\Box,k)$ Heffter ruler & 1st density & 2nd density   \\
\hline
\hline {$3$} & $67$ & $\{1,10,56\}$ & 0.1875 & 0.5     \\
\hline {$5$} & $71$   & Example \ref{35} & 0.5882... & 0.7647... \\
\hline {$7$} & $211$ &  $\{1,4,82,64,154,59,58\}$ & 0.4038... & 0.5384... \\
\hline {$9$} & $271$ & $\{1,36,110,44,179,56,224,156,7\}$  & 0.5373... & 0.6417... \\
\hline {$11$} & $419$ & $\{1,4,148,64,388,45,226,363,48,73,316\}$ & 0.5288... & 0.6153... \\
\hline {$13$} & $599$ & $\{1, 49, 515, 245, 181, 526, 117, 34, 332, 432, 130, 424, 9\}$  & 0.5217... & 0.5973... \\
\hline
\end{tabular}
\end{center}


\medskip
It is natural to say that two $(\F_q^\Box,k)$ Heffter rulers $A$ and $B$ are {\it equivalent} if we
have $At=B$ for a suitable $t\in\F_q$. In this case it is clear that their lists of differences $\Delta A$ and $\Delta B$
coincide.
Hence, for the existence of a $(\F_q^\Box,k;n)$ Heffter difference packing it is necessary to
have at least $n$ inequivalent $(\F_q,k)$ Heffter rulers. Let us show that in the special
case $k=3$ this condition is also sufficient.

\begin{prop}\label{inequivalent}
The maximum number of pairwise inequivalent $(\F_q^\Box,3)$ Heffter rulers coincide
with the maximum $n$ for which there exists a $(\F_q^\Box,3;n)$ Heffter difference packing.
\end{prop}
\begin{proof}
By Remark \ref{Heffterruler}, a $(\F_q^\Box,3;n)$ Heffter difference packing is the same as a set of $n$ $(\F_q^\Box,3)$ Heffter rulers
whose lists of differences are pairwise disjoint.
Thus, to prove the assertion, it is enough to show that two $(\F_q,3)$ Heffter rulers $A$ and $B$ are inequivalent if and only if
their lists of differences are disjoint, that it is the same as to show that they are equivalent if and only if
their lists of differences are not disjoint.

If $A$ and $B$ are equivalent, their lists are certainly not disjoint since they even coincide.

Now assume that $\Delta A$ and $\Delta B$ are not disjoint so that we have $a_1a^{-1}_2=b_1b^{-1}_2$
for suitable elements $a_1, a_2\in A$ and $b_1, b_2\in B$. It follows that $\{a_1,a_2\}\cdot t=\{b_1,b_2\}$ with $t=a_2^{-1}b_2$.
Thus $At$ and $B$ share the two elements $b_1$ and $b_2$.
Recall that both $A$ and $B$ are zero-sum so that $At$ is zero-sum as well.
It follows that we have $At=\{b_1,b_2,-b_1-b_2\}=B$, i.e., $A$ and $B$ are equivalent.
\end{proof}

In the following table we report the number $r$ of inequivalent $(\F^\Box_q,3)$ Heffter rulers
for all admissible $q<500$.
\small
\begin{center}
\begin{tabular}{|c||c|c|c|c|c|c|c|c|c|c|c|c|c|c|c|c|c|c|c|c|}
\hline {$q$} &  19&31&43&67&79&103&127&139&151&163&199&211&223&271&283&307&331&367&379   \\
\hline {$r$} &0&0&0&1&0&0&1&1&2&2&3&3&2&3&2&3&1&2&2  \\
\hline
\end{tabular}
\end{center}
\normalsize

\medskip
For $k>3$ Proposition \ref{inequivalent} is false; the number of inequivalent $(\F_q^\Box,k)$ Heffter rulers
is usually much greater than the maximum $n$ for which there exists a $(\F_q^\Box,k;n)$ Heffter difference packing.
As an example, we have checked that there are 26 inequivalent $(\F_{151}^\Box,5)$ Heffter rulers. Yet, there is no
triple of them forming a $(\F_{151}^\Box,5;3)$ Heffter difference packing.
A $(\F_{151}^\Box,5;2)$ Heffter difference packing is given by
 $$\bigl{\{}\{1,36,58,110,97\}, \{1,78,22,139,62\}\bigl{\}}.$$

\section{Some constructions for Heffter nets over a finite field}

In this section we aim to construct some $(k^2,k;r)$ Heffter nets over a finite field.
Thus $2k^2+1$ should be a prime power $q$ and the point set should be a half-set of $\F_q$.
In this case the method of differences considered in the previous section is totally impotent.
Indeed the application of this method would require a Heffter $(\F_q^\Box,k)$ ruler and hence,
by definition, a resolvable $(k^2,k)$-MGR. This would be a $k$-subset of $\Z_{k^2}$ whose list of
differences is $\Z_{k^2}\setminus k\Z_{k^2}$, i.e., a cyclic $(k,k,k,1)$ {\it relative difference set}
 whose non-existence has been known for a long time (see Theorem 6.2 in \cite{EB}).

\subsection{Constructing $(9n^2,3n;4)$ Heffter nets}
We first consider the case when $q=2k^2+1$ with $k$ divisible by 3, say $k=3n$, so that $q=18n^2+1$. The following theorem
gives a method to construct a $(9n^2,3n;4)$ Heffter net over $\F_q$ with $q=18n^2+1$ a prime power and $n$ odd. In the proof
we start from a standard $(9n^2,3n;4)$ Heffter net $\mathcal N$ with
point set $\Z_{3n}\times\Z_{3n}$ and then we show that under the hypothesis of the theorem there exists a Heffter labeling of $\mathcal N$
with labels in $\F_q$.

\begin{thm}\label{18n^2+1}
Let $q=18n^2+1$ be a prime power with $n$ odd, let $x$ be a $(3n)$-th primitive root of unity of $\F_q$,
and let $Y=(y_0,y_1,\dots,y_{3n-1})$ be a zero-sum complete system of representatives for the cosets
of $C^{3n}$ in $\F_q^*$ such that both the sums
$\displaystyle\sigma=\sum_{i=0}^{3n-1}x^{i}y_{i}$ and $\displaystyle\sigma'=\sum_{i=0}^{3n-1}x^{-i}y_{i}$ are null.
Then there exists a $(9n^2,3n;4)$ Heffter net.
\end{thm}
\begin{proof}
Consider the following four partitions of $\Z_{3n}\times\Z_{3n}$ into subsets of size $3n$:
$${\mathcal P}_1=\bigl{\{}\Z_{3n} \times \{j\} \ | \ j\in \Z_{3n}\bigl{\}};$$
$${\mathcal P}_2=\bigl{\{}\{(0,j),(1,j+1),(2,j+2),\dots,(3n-1,j+3n-1)\} \ | \ j\in \Z_{3n}\bigl{\}};$$
$${\mathcal P}_3=\bigl{\{}\{(0,j),(1,j-1),(2,j-2),\dots,(3n-1,j-3n+1)\} \ | \ j\in \Z_{3n}\bigl{\}};$$
$${\mathcal P}_4=\bigl{\{}\{i\} \times \Z_{3n} \ | \ i\in \Z_{3n}\bigl{\}}.$$
It is  readily seen that  that the triple ${\mathcal N}=(V,{\mathcal B},{\mathcal R})$ with
$$V=\Z_{3n}\times\Z_{3n},\quad\quad {\mathcal B}=\displaystyle\bigcup_{i=1}^4{\mathcal P}_i,\quad\quad {\rm and}\quad\quad
{\mathcal R}=\{{\mathcal P}_1,{\mathcal P}_2,{\mathcal P}_3,{\mathcal P}_4\}$$
is a $(9n^2,3n;4)$ net. Thus, to prove the assertion it is enough to show that $\mathcal N$ admits a Heffter labeling.

Let $X=\langle x\rangle$ be the group of $(3n)$-th roots of unity of $\F_q$. Considering that $n$ is odd, $\{1,-1\}\cdot X$ is the group
of $(6n)$-th roots of unity. Note that this group is the same as the group $C^{3n}$ of $(3n)$-th powers of $\F_q^*$
so that we can write
$$\{1,-1\}\cdot X=C^{3n}\quad\quad{\rm and}\quad\quad C^{3n}\cdot Y=\F_q^*.$$
Thus, setting $V=X\cdot Y$, we see that we have $\{1,-1\}\cdot V=\F_q^*$, i.e., $V$
is a half-set of $\F_q$. Now consider the map
$f: (i,j)\in \Z_{3n}\times\Z_{3n} \longrightarrow x^iy_j\in V$
which is clearly bijective; $x^{i_1}y_{j_1}=x^{i_2}y_{j_2}$ implies $y_{j_1}y_{j_2}^{-1}=x^{i_2-i_1}\in C^{3n}$
so that $j_1=j_2$ by definition of $Y$ and then $i_1=i_2$.
Let us prove that $f$ is a Heffter labeling, namely that $f(B)$ is zero-sum for every block $B$ of the net considered above.

If $B\in{\mathcal P}_1$, then $B=\{(i,j) \ | \ 0\leq i\leq 3n-1\}$ for some $j$, hence we have
$$\sum_{P\in B}f(P)=\sum_{i=0}^{3n-1}x^iy_j=y_j\sum_{i=0}^{3n-1}x^i=0$$
since $X$ is zero-sum. Indeed, as recalled earlier, any non-trivial subgroup
of the multiplicative group of a finite field is zero-sum (see, e.g., Fact 1 in \cite{BN1}).

If $B\in{\mathcal P}_2$, then $B=\{(i,j+i) \ | \ 0\leq i\leq 3n-1\}$ for some $j$, hence we have
$$\sum_{P\in B}f(P)=\sum_{i=0}^{3n-1}x^iy_{j+i}=x^{-j}\sum_{i=0}^{3n-1}x^{j+i}y_{j+i}=x^{-j}\sigma=0$$
since $\sigma$ is null by assumption.

If $B\in{\mathcal P}_3$, then $B=\{(i,j-i) \ | \ 0\leq i\leq 3n-1\}$ for some $j$, hence we have
$$\sum_{P\in B}f(P)=\sum_{i=0}^{3n-1}x^iy_{j-i}=x^{j}\sum_{i=0}^{3n-1}x^{i-j}y_{j-i}=x^{j}\sigma'=0$$
since $\sigma'$ is null by assumption.

Finally, if $B\in{\mathcal P}_4$, then $B=\{(i,j) \ | \ 0\leq j\leq 3n-1\}$ for some $i$, hence we have
$$\sum_{P\in B}f(P)=\sum_{j=0}^{3n-1}x^iy_{j}=x^{i}\sum_{i=0}^{3n-1}y_{j}=0$$
since $Y$ is zero-sum by assumption.
The assertion follows.
\end{proof}

We are not able to say whether the above theorem assures the existence of infinitely many $(9n^2,3n;4)$ Heffter nets
for the simple reason that, as far as we are aware, it is not known whether the integer sequence $18n^2+1$ contains
infinitely many prime powers. We are not even able to say whether the theorem guarantees the existence of a $(9n^2,3n;4)$ Heffter
net for ``many" values of $n$. Anyway, a simple probabilistic argument make us believe that the theorem succeeds for any admissible $n>1$. Indeed
the number of $(3n)$-tuples $Y$ of pairwise distinct representatives for the cosets of $C^{3n}$ in $\F_{q}^*$, that is $(3n)!(6n)^{3n}$,
is huge compared with the number $(18n^2+1)^3$ of triples of $\F_q^3$.
Thus there is a high probability that there is at least one $Y$ whose related triple $(\sum_{y\in Y}y,\sigma,\sigma')$ is null.

In the following detailed example we apply
the above theorem with $n=3$.
\begin{ex}
We can apply Theorem \ref{18n^2+1} with $n=3$ since $q=18\cdot3^2+1=163$ is a prime.
As a $9$-th primitive root of unity of $\Z_{163}$ we can take $x=40$, hence
the group of $9$-th roots of unity of $\Z_{163}$ is
$$X=\{1,x,\dots,x^8\}=\{1, 40, 133, 104, 85, 140, 58, 38, 53\}.$$
Now consider the following $9$-subset of $\Z_{163}$:
$$Y=\{y_0,y_1,\dots,y_8\}=\{1, 2, 160, 142, 119, 84, 36, 128, 143\}.$$
It can be checked that $2$ is a primitive element of $\Z_{163}$ and that $y_i=2^{\alpha_i}$ with $\alpha_i\equiv i$ (mod 9)
for each $i$ as follows:
$$(\alpha_0,\dots,\alpha_8)=(0, 1, 20, 93, 130, 14, 42, 7, 98).$$
This guarantees that $Y$ is a complete system of representatives for the cosets of $C^9$ in $\Z_{163}^*$.
One can also check that we have
$$\sum_{i=0}^8y_i = 0, \quad \sigma=\sum_{i=0}^8x^iy_i=0,\quad {\rm and}\quad
\sigma'=\sum_{i=0}^8x^{-i}y_i=0.$$ Thus there exists a $(81,9;4)$ Heffter net.
We note that the four parallel classes of this Heffter net are the rows, the columns, the {\it right diagonals} and the
 {\it left diagonals} of the matrix $(x^iy_j)$ below.
$$\begin{pmatrix}
1& 2& 160& 142& 119& 84& 36& 128& 143\cr
40& 80& 43& 138& 33& 100& 136& 67& 15\cr
133& 103& 90& 141& 16& 88& 61& 72& 111\cr
104& 45& 14& 98& 151& 97& 158& 109& 39\cr
85& 7& 71& 8& 9& 131& 126& 122& 93\cr
140& 117& 69& 157& 34& 24& 150& 153& 134\cr
58& 116& 152& 86& 56& 145& 132& 89& 144\cr
38& 76& 49& 17& 121& 95& 64& 137& 55\cr
53& 106& 4& 28& 113& 51& 115& 101& 81
\end{pmatrix}$$
We also note that this matrix is a rank-one H$(9,9)$ (see \cite{B}).
\end{ex}

We found a $(9n^2,3n;4)$ Heffter net via Theorem \ref{18n^2+1} and the help of a computer also for $n=7$ and $n=9$.

For $n=7$, we have $q=18n^2+1=883$ and a $(441,21;4)$ Heffter net over $\F_q$  can be obtained
via Theorem \ref{18n^2+1} taking
$$x=729\quad{\rm and}\quad Y=(2^0, 2^1, 2^2, \dots, 2^{13}, 490, 97, 60, 72, 483, 680, 278).$$

For $n=9$, we have $q=18n^2+1=1459$ and a $(729,27;4)$ Heffter net over $\F_q$ can be obtained
via Theorem \ref{18n^2+1} taking
$$x=1080\quad{\rm and}\quad Y=(3^0, 3^1, 3^2, \dots, 3^{19}, 546, 597, 652, 1307, 1386, 467, 1338).$$

\subsection{A $(121,11;9)$ Heffter net}
Now we should consider the case when $q=2k^2+1$ with $k$ not divisible by 3. Here we have $k\equiv\pm1$ (mod 3) and then $q\equiv0$ (mod 3)
so that $q$ is a power of 3. It follows that $q=3^n=2k^2+1$ for some $n$. Given that this identity can be rewritten as ${3^n-1\over3-1}=k^2$
and considering the known results on the Nagell-Ljunggren equation ${x^n-1\over x-1}=y^m$ already mentioned in Section 2,
we necessarily have $q=3^5$ and $k=11$. On the one hand this is disappointing since it means that for $k\not\equiv0$ (mod 3) the only attempt to construct a $(k^2,k;r)$ Heffter net over
a finite field is to try with a $(121,11;r)$ Heffter net over $\F_{3^5}$ for some $r$.
On the other hand this attempt led us to the Heffter space that surprised us the most. Indeed it is a $(121,11;9)$ Heffter net which, with its remarkable density
of ${3\over4}$, is the most dense Heffter configuration that we have at this moment.
Our strategy for constructing this net was the following. We started from the affine plane of order 11 which, of course, is a $(121,11;12)$ net.
Since a $(121,11;12)$ Heffter net cannot exist by Proposition \ref{linear}, we deleted $d$ parallel classes of the plane in various ways getting some
$(121,11;12-d)$ nets. Then we tried to obtain a Heffter labeling of at least one of these nets with labels in $\F_{3^5}$. We failed with $d=1,2$ but we succeeded with $d=3$.

\eject
\begin{thm}
There exists a $(121,11;9)$ Heffter net over $\F_{3^5}$.
\end{thm}
\begin{proof}
The affine plane AG$(2,11)$ is clearly a $(121,11;12)$ net
with point set $\Z_{11}\times\Z_{11}$ and resolution ${\mathcal R}=\{{\mathcal P}_s \ | \ s\in\Z_{11} \ \cup \ \{\infty\}\}$ where ${\mathcal P}_s$
is the set of lines of slope $s$:
$${\mathcal P}_s=\bigl{\{} \{(i,si+j) \ | \ 0\leq i\leq10\} \ | \ 0\leq j\leq 10\bigl{\}}\quad\quad\mbox{for $0\leq s\leq 10$};$$
$${\mathcal P}_\infty=\bigl{\{} \{i\}\times\Z_{11} \ | \ 0\leq i\leq 10\bigl{\}}.$$
Given any subset $S$ of $\Z_{11} \ \cup \ \{\infty\}$, set ${\mathcal B}(S)=\bigcup_{s\in S}{\mathcal P}_s$
and ${\mathcal R}(S)=\{{\mathcal P}_s \ | \ s\in S\}$.
It is clear that the triple ${\mathcal N}(S)=(\Z_{11}\times\Z_{11},{\mathcal B}(S),{\mathcal R}(S))$
is a  $(121,11;|S|)$ net. Thus, to prove the assertion, we have to find a 9-subset $S$ of $\Z_{11} \ \cup \ \{\infty\}$
for which ${\mathcal N}(S)$ admits a Heffter labeling.

Let $g$ be a root of the primitive polynomial $z^5+2z^4+1$ over $\F_3$ so that
$\F_{3^5}^*=\{g^i \ | \ 0\leq i\leq 242\}$.
Consider the 11-tuple
$$Y=(y_0,y_1,y_2,\dots,y_{10})=(g^0, g^{1}, g^{18}, g^{3}, g^{81}, g^{27}, g^{54}, g^{162}, g^{6}, g^{9}, g^{2}).$$
Note that the map $g^i\in\F_q^*\longrightarrow i\in\Z_{11}$ is bijective on $Y$:
$$(0,1,18,3,81,27,54,162,6,9,2)\equiv(0,1,7,3,4,5,10,8,6,9,2) \quad \mbox{(mod 11)}.$$
Thus $Y$ is a complete system of representatives for the cosets of $C^{11}$ in $\F_{3^5}^*$
so that we have \begin{equation}\label{Y}C^{11}\cdot Y=\F_{3^5}^*\end{equation}
Now consider the group $X$ of 11-th roots of unity which can be also viewed as the group $C^{22}$ of 22nd powers
so that we have $X=\langle x\rangle$ with $x=g^{22}$.
Clearly, $\{1,-1\}\cdot X$ is the group of 22nd roots of unity which can be also viewed as the group $C^{11}$ of 11th powers:
\begin{equation}\label{X}\{1,-1\}\cdot X=C^{11}.\end{equation}
Putting (\ref{Y}) and  (\ref{X}) together we get
\begin{equation}\label{XY}\{1,-1\}\cdot X\cdot Y=\F_{3^5}^*\end{equation}
which means that $V:=X\cdot Y$ is a half-set of $\F_{3^5}$.

Now, taking into account that $g^5=g^4+2$ by definition of $g$, one can check that the following identities hold
\begin{equation}\label{identities}
\sum_{i=0}^{10}y_{i}=0 \quad\quad{\rm and}\quad\quad \sum_{i=0}^{10}x^{i}y_{si}=0\quad \mbox{for $s\in\{1,2,3,5,7,9,10\}$}
\end{equation}
where the indices have to be understood modulo 11.

Consider the bijective map $$f: (i,j)\in \Z_{11}\times\Z_{11} \longrightarrow x^iy_j\in V$$
and let us calculate the sums $\sum_{P\in L}f(P)$ with $L$ a line of AG$(2,11)$ of slope
$s\in S:=\{0,1,2,3,5,7,9,10,\infty\}$.

If $L$ is a line of ${\mathcal P}_0$,  then $L=\{(i,j) \ | \ 0\leq i\leq10\}$ for some $j$ so that we have
$$\sum_{P\in L}f(P)=\sum_{i=0}^{10}x^{i}y_{j}=y_j\sum_{i=0}^{10}x^i=0$$
since $X$ is zero-sum (see again Fact 1 in \cite{BN1}).

If $L$ is a line of ${\mathcal P}_\infty$,  then $L=\{(i,j) \ | \ 0\leq j\leq10\}$ for some $i$ so that we have
$$\sum_{P\in L}f(P)=\sum_{j=0}^{10}x^{i}y_{j}=x_i\sum_{j=0}^{10}y^j=0$$
since (\ref{identities}) guarantees that $Y$ is zero-sum.

Now let $L$ be a line of ${\mathcal P}_s$ with $s\in\{1,2,3,5,7,9,10\}$. We have $L=\{(i,si+j) \ | \ 0\leq i\leq10\}$ for some $j$.
Set $t=s^{-1}j$ and note that we have
$$\sum_{P\in L}f(P)=\sum_{i=0}^{10}x^{i}y_{si+j}=x^{-t}\sum_{i=0}^{10}x^{t+i}y_{s(t+i)}=x^{-t}\sum_{i=0}^{10}x^{i}y_{si}=0$$
in view of (\ref{identities}).
We conclude that the map $f$ is a Heffter labeling of the $(121,11;9)$-net ${\mathcal N}(S)$ with
$S=\{0,1,2,3,5,7,9,10,\infty\}$. The assertion follows.
\end{proof}

The above $(121,11;9)$ Heffter net can be displayed by means of the following rank-one $11\times11$ matrix $M=(m_{i,j})$ over $\F_{3^5}$
where $m_{i,j}:=x^iy_j$ has been calculated taking into account again the basic identity $g^5=g^4+2$.
\scriptsize
$$
\begin{pmatrix}
10000& 0 1000& 1 1 0 1 0& 0 0 0 1 0& 1 1 0 0 1& 1 2
  1 0 2& 0 2 2 0 0& 1 1 0 2 2& 2 2 0 0 1& 2 2 2 2 0& 0 0 100\cr
  1 2 2 0 0& 0 1 2 2 0& 2 0 1 0 1& 1 0 0 1 1&
0 1 1 2 2& 2 0 1 0 0& 0 2 0 2 1& 1 2 1 1 0&
1 1 2 1 2& 1 0 1 1 0& 0 0 1 2 2\cr
1 1 2 2 1& 2 1 1 2 0& 1 2 2 2 1& 1 0 2 1 0&
0 0 0 0 2& 2 1 2 2 2& 1 0 1 0 1& 2 1 1 1 0&
0 2 0 1 0& 2 2 0 0 0& 0 2 1 1 2\cr
2 1 0 2 2& 1 2 1 0 1& 2 2 2 1 2& 2 2 1 2 2&
1 2 0 0 1& 2 1 2 2 1& 0 0 0 2 1& 0 2 1 2 0&
1 2 1 2 1& 2 0 2 1 0& 2 1 2 1 1\cr
2 1 0 1 0& 0 2 1 0 1& 1 2 1 0 0& 2 2 0 2 2&
0 2 0 1 2& 0 0 2 2 2& 1 1 0 2 1& 2 2 2 2 1&
2 2 1 2 0& 0 1 0 2 2& 2 0 2 1 1\cr
0 2 0 0 1& 2 0 2 0 1& 1 1 1 0 2& 2 2 2 0 0&
2 1 1 2 2& 0 2 2 0 1& 2 1 1 1 1& 0 1 1 1 2&
1 0 0 2 1& 0 0 2 1 0& 2 2 0 2 1\cr
2 0 1 1 2& 1 2 0 1 0& 2 2 2 1 0& 2 0 1 2 1&
2 1 1 0 2& 2 0 0 2 0& 2 0 1 2 2& 2 0 0 2 1&
2 0 2 2 1& 1 0 2 2 2& 0 1 2 0 1\cr
1 0 2 0 1& 2 1 0 2 1& 0 0 1 0 2& 0 2 2 1 0&
0 1 1 1 0& 1 1 1 2 2& 2 0 2 1 2& 0 2 1 2 1&
0 2 0 0 2& 1 1 1 0 1& 2 2 1 0 0\cr
0 0 1 1 0& 0 0 0 1 1& 1 2 1 2 0& 1 2 0 0 2&
1 1 0 2 0& 1 2 2 0 2& 1 0 0 2 0& 1 0 2 2 0&
1 1 1 1 1& 0 1 2 1 1& 2 0 0 0 2\cr
1 0 1 0 0& 0 1 0 1 0& 0 1 1 2 1& 2 0 0 1 1&
0 0 1 1 2& 2 0 2 2 2& 0 2 2 2 2& 0 2 1 0 0&
1 1 2 2 2& 0 2 1 1 1& 0 0 1 0 1\cr
1 2 0 2 2& 1 1 2 0 1& 1 2 0 0 0& 2 2 1 1 0&
2 2 1 0 1& 2 0 0 0 1& 0 1 0 1 1& 0 2 2 0 2&
1 2 0 2 1& 0 0 2 1 2& 2 1 1 2 1
  \end{pmatrix}
  $$

\normalsize
The point set is simply the set of all entries appearing in $M$. The parallel classes
are the rows, the columns and all the transversals having slope belonging
to the set $\{1,2,3,5,7,9,10\}$. By transversal of slope $u$ we mean any transversal
whose elements are the entries $m_{i,ui+j}$ with $0\leq i\leq 10$ and $j$ fixed.

Note that the 1st column is the group $X$ of 11th roots of unity and
that the 1st  row is the complete system of representatives $Y$
for the cosets of the 11th powers.

It might seem that there is a vague connection between the above net and the $(243,22,11,20)$ {\it strong external difference family} (SEDF)
constructed in \cite{WYFF} (see also \cite{JL}).
If $M^1$, $M^2$, \dots, $M^{11}$ are the columns of the matrix $M$ above and we set $B_i = M^i \ \cup \ (-M^i)$ for
$1\leq i\leq 11$, then $B_1, B_2, ..., B_{11}$ are precisely the blocks
of the mentioned SEDF.
However the rows and the columns of $M$ are {\it ordered} subsets of $\F_{3^5}$, not simply
subsets.

\medskip
The Heffter nets constructed in this section allow us to obtain immediately Theorem \ref{MOHSnet}
in view of Proposition \ref{MOHS=HS}.

\section{From Heffter spaces to mutually orthogonal cycle systems}

It is known that Heffter arrays are strongly related with graph decompositions obtainable via difference methods
(see Section 5 of \cite{DP} for the details). Here we will use Heffter spaces to construct sets of {\it mutually orthogonal cycle systems}.
First, we need to introduce some background on this topic.

As it is standard, speaking of a $k$-cycle $(x_0,x_1,\dots,x_{k-1})$ where the $x_i$'s
are pairwise distinct elements of any set $X$, we will mean the graph whose vertices are $x_0, x_1, \dots, x_{k-1}$ and whose
edges are $\{x_0,x_1\}$, $\{x_1,x_2\}$, \dots, $\{x_{k-1},x_0\}$.

A {\it $k$-cycle system of order $2v+1$} is a set $\mathcal C$ of $k$-cycles whose edges partition the edge set of $K_{2v+1}$, the complete graph on $2v+1$ vertices.
Such a system is said to be {\it $G$-regular} if the vertex set of $K_{2v+1}$ is an additive group $G$ and we have
$C+g\in{\mathcal C}$ for every pair $(C,g)\in {\mathcal C}\times G$. In other words if, up to isomorphism, $G$ is an
{\it automorphism group} of $\mathcal C$.

A 3-cycle system is usually called a {\it Steiner triple system} and one usually denotes it by STS$(2v+1)$ if its order is $2v+1$.
It is evident that it is equivalent to a S$(2,3,2v+1)$, this is why in this case the vertices of $K_{2v+1}$ are usually called {\it points} and the
3-cycles are called {\it blocks} or {\it triples}.
The literature on Steiner triple systems is huge and we refer to \cite{CR} for a comprehensive textbook on the topic.

The existence problem for a $k$-cycle system of order $2v+1$ has been completely solved in \cite{AG} for $k$ odd and in \cite{Sajna}
for $k$ even. See also \cite{v<3k} for an alternative purely algebraic proof for the odd case.

Cycle systems are special {\it graph decompositions}. Given a subgraph $\Gamma$ of a graph $K$, a $\Gamma$-decomposition of $K$ is a set $\mathcal D$
of graphs isomorphic to $\Gamma$ whose edges partition $E(K)$. Thus a $k$-cycle system of order $2v+1$ is a $\Gamma$-decomposition of $K_{2v+1}$ where $\Gamma$
is the $k$-cycle. In general, two graph decompositions $\mathcal D$ and $\mathcal D'$ of the same graph $K$ are said to be {\it orthogonal} if any graph of $\mathcal D$
has at most one edge in common with any graph of $\mathcal D'$ (see, e.g., \cite{CY})
and a multitude of authors actually adhered to this definition.
Yet, concerning the orthogonality of $k$-cycle systems, we think it is appropriate to distinguish the case $k=3$ from the case $k>3$
in order to put some order in the literature.

\subsection{Orthogonal Steiner triple systems}

According to the general definition of orthogonal graph decompositions, two Steiner triple systems should be said orthogonal
precisely when they do not have any triple in common, i.e., they are disjoint. So it was done, for instance, in the recent papers \cite{BCP,KY}.
On the other hand this is in contrast
with an extensive literature on {\it orthogonal Steiner triple systems} as introduced in \cite{shaug}.
In this literature, two {\it Steiner triple systems} ${\mathcal C}$ and ${\mathcal C}'$
are defined to be orthogonal if, besides being disjoint,  satisfy the following demanding condition:
\begin{equation}\label{superorthogonal}
\{u,v,a\}, \{x,y,a\} \in {\mathcal C} \quad{\rm and}\quad \{u,v,w\}, \{x,y,z\} \in {\mathcal C}' \Longrightarrow w\neq z.
\end{equation}
Here, just in order to avoid confusion, we say that two STSs are orthogonal when they are disjoint (hence orthogonal in the usual sense
used for graph decompositions), and that they are {\it super-orthogonal}
when they are orthogonal in the sense of \cite{shaug}.
We prefer, however, to re-elaborate condition (\ref{superorthogonal}) still in terms of orthogonality in the general sense as follows.
Observe that if ${\mathcal C}$ is a STS$(v)$ with point set $V$ and $x$ is one of its points, then the set $N(x)$ of all pairs $\{y,z\}$ such that $\{x,y,z\}$ is a triple of $\mathcal C$
is a near 1-factor of $K_v$. Also, the set ${\mathcal N}=\{N(x) \ | \ x\in V\}$ is a near 1-factorization of $K_v$
that we call {\it the near 1-factorization associated with $\mathcal C$}.
It is then clear that the definition of super-orthogonality can be equivalently reformulated as follows.
\begin{defn}
Two Steiner triple systems are super-orthogonal when they do not have any triple in common and their associated near 1-factorizations
are orthogonal.
\end{defn}

In view of the well celebrated results about the existence of {\it large sets of Steiner triple systems} \cite{Lu1,Lu2,T}
we can say that the maximum number of mutually orthogonal STS$(v)$ is $v-2$ whichever is $v$.

Concerning the maximum number $\omega(v)$ of mutually super-orthogonal STS$(v)$ we have the following:
\begin{itemize}
\item $\omega(v)=1$ for $v=3,9$ (see \cite{MN});

\item $\omega(v)=2$ for $v=7,13, 15$ (see \cite{Gib});

\item $\omega(v)\geq2$ for all admissible $v\geq19$  (see \cite{CGMMR});

\item $\omega(v)\geq3$ for all admissible $v\geq19$ with 24 possible exceptions (see \cite{DDL}).

\item $\omega(v)>n$ for $v$ a prime power sufficiently large with respect to $n$ (see \cite{Gross}).
\end{itemize}

\subsection{Orthogonal $k$-cycle systems with $k>3$}

For $k>3$, according to the general definition of orthogonal graph decompositions,
two $k$-cycle systems ${\mathcal C}$, ${\mathcal C}'$ of the same order $2v+1$ are orthogonal if
any cycle  $C\in {\mathcal C}$ shares at most one edge with any cycle $C'\in {\mathcal C}'$.

Apart from special instances of the pair $(k,2v+1)$, to establish the maximum number $\mu(k,2v+1)$ of mutually orthogonal $k$-cycle systems of order $2v+1$
appears at the moment quite hard. Even to establish a non-trivial lower bound appears to be difficult.

For now, it is known that we have $\mu(k,2kn+1)\geq2$ in the following cases:
\begin{itemize}
\item $4\leq k\leq 10$ and $kn\equiv0$ or 3 (mod 4) (see \cite{CMPP});

\item $k=5$ and $3\leq n\leq 100$ (see \cite{DM});

\item $k\equiv0$ (mod 4) (see \cite{BDCY});
\item $k\equiv3$ (mod 4) and $n\equiv1$ (mod 4) (see \cite{BDCY});
\item $k\equiv3$ (mod 4) and a sufficiently large $n\equiv0$ (mod 4) (see \cite{BDCY}).
\end{itemize}

We also have:
\begin{itemize}
\item $\mu(k,v)\geq2$ for $4\leq k\leq 9$ and any admissible $v$ (see \cite{KY}).
\end{itemize}

A greater lower bound has  been established for $k$ even in \cite{BCP}:
\begin{itemize}
\item $\mu(k,2kn+1)\geq \begin{cases}4n & {\rm if} \ k=4 \medskip\cr {n\over 4k-2}-1 & {\rm if} \ k\equiv0 \ {\rm (mod \ 4)}, \ k>4 \medskip\cr {n\over 24k-18}-1 & {\rm if} \ k\equiv2 \
{\rm (mod \ 4)}\end{cases}$
\end{itemize}



\medskip
Here, exploiting the results of the previous sections, we will be able to prove the existence of
an arbitrarily large set of mutually orthogonal $k$-cycle systems for any odd $k$.
For this we need again some difference methods.

The {\it list of differences} of a graph $\Gamma$ with vertices in an additive group $G$ is the multiset
$$\Delta \Gamma=\bigcup_{e\in E(\Gamma)}\Delta e$$ where $\Delta e$ is the list of differences of the
edge $e$ as defined in Section 4: if $e=\{x,y\}$, then $\Delta e=\{x-y,y-x\}$.
In other words, $\Delta \Gamma$ is the list of all possible differences between two adjacent vertices of $\Gamma$.

Note, in particular, that a $k$-cycle $C=(c_0,c_1,\dots,c_{k-1})$ with vertices in a group $G$ has list of differences
$\Delta C=\{c_{i+1}-c_{i}, \ c_{i}-c_{i+1}  \ | \ 0\leq i\leq k-1\}$ where it is understood that $c_k=c_0$.

The following basic fact is just a special consequence of a more general result concerning
{\it $G$-regular graph decompositions} obtainable via difference methods (see, e.g., Theorem 2.1 in \cite{BP2}).
\begin{prop}\label{DF}
If  $G$ is an additive group of order $2kn+1$, then the existence of a $G$-regular $k$-cycle system is equivalent to a $n$-set ${\mathcal F}=\{C_1,\dots,C_n\}$
of $k$-cycles with vertices in $G$ such that  $\Delta C_1 \ \cup \ \dots \ \cup \ \Delta C_n=G\setminus\{0\}$.
\end{prop}
A set $\mathcal F$ as in the statement of Proposition \ref{DF} is said to be a set of {\it base cycles} for the
related $G$-regular $k$-cycle system $\mathcal C$ since we have ${\mathcal C}=\{C_i+g \ | \ 1\leq i\leq n; g\in G\}$.


Recall that $B=\{b_0,b_1,\dots,b_{k-1}\}$ is a simple subset of $G$ (see Definition \ref{simple}) if the associated sequence of partial
sums $C=(c_0,c_1,\dots,c_{k-1})=(b_0,b_0+b_1,b_0+b_1+b_2,...)$ does not have repeated elements. In this case $C$ may be interpreted as a $k$-cycle and
noting that $c_{i}-c_{i-1}=b_{i}$ for $1\leq i\leq k-1$ we can write
$$\Delta C=\pm\{b_1,b_2,\dots,b_{k-1},c_{k-1}-b_0\}.$$
Thus, considering that $c_{k-1}$ is the sum of all the elements of $B$, when $B$ is zero-sum we have $\Delta C=B \ \cup \ -B$.

Hence, given that the blocks of a Heffter system are all zero-sum, if ${\mathcal P}=\{B_1,\dots,B_n\}$ is a simple $(v,k)$
Heffter system on a half-set $V$ of $G$, then we have
$$\Delta C_1 \ \cup \ \dots \ \cup \ \Delta C_n=(B_1 \ \cup \ \dots \ \cup \ B_n) \ \cup \ -(B_1 \ \cup \ \dots \ \cup \ B_n)$$
where $C_i$ is the ``partial sum cycle" associated with $B_i$ for $1\leq i\leq n$.
At this point, recalling that the blocks of $\mathcal P$ partition $V$ and that $V$ is a half-set of $G$, we see that
$$\Delta C_1 \ \cup \ \dots \ \cup \ \Delta C_n=V \ \cup \ -V=G\setminus\{0\}.$$
Thus $\{C_1,\dots,C_n\}$ is a set of base cycles for a $G$-regular $k$-cycle system in view of Proposition \ref{DF}.


So we have the following.
\begin{prop}\label{SHS->CS}
If ${\mathcal P}$ is a simple $(v,k)$ Heffter system over $G$, then the partial-sum cycles associated with the blocks
of $\mathcal P$ form a set of base cycles for a $G$-regular $k$-cycle system.
\end{prop}

One could show that the converse is also true when $G$ has order $2kn+1$ for some $n$: in this case any $G$-regular $k$-cycle system
is derived from a $(v,k)$ Heffter system over $G$. This is why we can say that a $(v,k)$ Heffter system
over $\Z_{2v+1}$ exists: as recalled in the introduction,
the existence of a $\Z_{2kn+1}$-regular $k$-cycle system has been established for all possible pairs $(v,k)$ (see \cite{BurDel} and \cite{BGL}).

Let ${\mathcal F}$ and ${\mathcal F}'$ be the sets of  base cycles arising, via Proposition \ref{SHS->CS}, from two orthogonal
simple Heffter systems ${\mathcal P}$ and ${\mathcal P}'$, respectively. Also, let ${\mathcal C}$ and ${\mathcal C}'$ be
the respective $k$-cycle systems arising from them.
Assume that there is a cycle of $\mathcal C$, say $C+g$ with $(C,g)\in{\mathcal F}\times G$, sharing
two edges $e_1$, $e_2$ with a cycle of ${\mathcal C}'$, say $C'+h$ with $(C',h)\in{\mathcal F}'\times G$.
Then $\Delta e_1 \ \cup \ \Delta e_2$ is contained both in $\Delta(C+g)=\Delta C$ and in
$\Delta(C'+h)=\Delta C'$. Now recall that we have $\Delta C=B \ \cup \ -B$ where $B$ is the block of ${\mathcal P}$
from which ${\mathcal C}$ derives and, analogously, $\Delta C'=B' \ \cup \ -B'$ where $B'$ is the block of ${\mathcal P}'$
from which ${\mathcal C}'$ derives. Thus we have $\Delta e_1 \ \cup \ \Delta e_2\subset (B \ \cap \ B') \ \cup \ -(B \ \cap \ B')$.
Then two of the four elements of $\Delta e_1 \ \cup \ \Delta e_2$ necessarily are in $B \ \cap \ B'$ contradicting that
any block of ${\mathcal P}$ intersects any block of ${\mathcal P}'$ in at most one element.

We conclude that  ${\mathcal C}$ and ${\mathcal C}'$ are orthogonal. Thus we can state the following.
\begin{prop}\label{SMOHS->MOCS}
If there exists a MOHS$(v,k;r)$ whose members are all simple, then $\mu(k,2v+1)\geq r$.
\end{prop}



Taking into account Proposition \ref{SMOHS->MOCS} and the constructions for MOHSs obtained in the previous sections we can get the following.
\begin{thm} \label{orthogonalcycles}
We have:
\begin{itemize}
\item[(i)]$\mu(k,2v+1)\geq k$ for $(k,2v+1)\in\{(5,71),(7,211),(11,419),(13,599)\}$;
\item[(ii)] $\mu(5,151)\geq10$;
\item[(iii)] $\mu(3n,18n^2+1)\geq4$ for $n\in\{3,7,9\}$;
\item[(iv)] $\mu(11,243)\geq 9$;
\item[(v)] $\mu(k,2kw+1)\geq \lceil{w\over4k^4}\rceil$ whenever $2kw+1$ is a prime power and $kw$ is odd.
\end{itemize}
\end{thm}
\begin{proof}
For (i), it enough to use the simple $(\F_{2v+1}^\Box,k)$ Heffter rulers exhibited after Corollary \ref{asymptotic3}.
For (ii) we have to use the $(\F_{151}^\Box,5;2)$ Heffter difference packing considered at the end of the fourth section.
For (iii) and (iv) it is enough to consider the Heffter nets constructed in Section 5 checking that they are all simple.
Finally, (v) is a consequence of Corollary \ref{asymptotic2}.
\end{proof}

Note, in particular, that Theorem \ref{final} is nothing but Theorem \ref{orthogonalcycles}(v).

\begin{ex} Using the $(35,5;5)$-MOHS of Example \ref{35} we get five mutually orthogonal pentagon systems ${\mathcal C}_1$, \dots, ${\mathcal C}_5$ of order 71.
The related sets of base cycles ${\mathcal F}_1$, \dots, ${\mathcal F}_5$ are displayed below.
\small
\medskip\noindent
\begin{center}
\begin{tabular}{|l|c|r|c|r|c|r|c|r|c|r|c|r|}
\hline {\quad\quad\quad$\mathcal {F}_1$}    \\
\hline $\{1, 25, 50, 22, 0\}$\\
\hline $\{45, 60, 49, 67, 0\}$  \\
\hline $\{37, 2, 4, 33, 0\}$\\
\hline $\{32, 19, 38, 65, 0\}$  \\
\hline $\{20, 3, 6, 14, 0\}$\\
\hline $\{48, 64, 57, 62, 0\}$ \\
\hline $\{30, 40, 9, 21, 0\}$ \\
\hline
\end{tabular}\quad
\begin{tabular}{|l|c|r|c|r|c|r|c|r|c|r|c|r|}
\hline {\quad\quad\quad$\mathcal{F}_2$}    \\
\hline $\{49, 18, 36, 13, 0\}$  \\
\hline $\{4, 29, 58, 17, 0\}$\\
\hline $\{38, 27, 54, 55, 0\}$ \\
\hline $\{6, 8, 16, 61, 0\}$\\
\hline $\{57, 5, 10, 47, 0\}$  \\
\hline $\{9, 12, 24, 56, 0\}$\\
\hline $\{50, 43, 15, 35, 0\}$\\
\hline
\end{tabular}\quad
\begin{tabular}{|l|c|r|c|r|c|r|c|r|c|r|c|r|}
\hline {\quad\quad\quad$\mathcal{F}_3$}    \\
\hline $\{58, 30, 60, 69, 0\}$\\
\hline $\{54, 1, 2, 52, 0\}$\\
\hline $\{16, 45, 19, 68, 0\}$ \\
\hline $\{10, 37, 3, 7, 0\}$  \\
\hline $\{24, 32, 64, 31, 0\}$\\
\hline $\{15, 20, 40, 46, 0\}$\\
\hline $\{36, 48, 25, 11, 0\}$\\
\hline
\end{tabular}\quad
\begin{tabular}{|l|c|r|c|r|c|r|c|r|c|r|c|r|}
\hline {\quad\quad\quad$\mathcal{F}_4$}    \\
\hline $\{2, 50, 29, 44, 0\}$\\
\hline $\{19, 49, 27, 63, 0\}$\\
\hline $\{3, 4, 8, 66, 0\}$ \\
\hline $\{64, 38, 5, 59, 0\}$  \\
\hline $\{24, 32, 64, 31, 0\}$\\
\hline $\{15, 20, 40, 46, 0\}$\\
\hline $\{36, 48, 25, 11, 0\}$\\
\hline
\end{tabular}\quad
\begin{tabular}{|l|c|r|c|r|c|r|c|r|c|r|c|r|}
\hline {\quad\quad\quad$\mathcal{F}_5$}    \\
\hline $\{58, 30, 60, 69, 0\}$\\
\hline $\{54, 1, 2, 52, 0\}$\\
\hline $\{16, 45, 19, 68, 0\}$ \\
\hline $\{10, 37, 3, 7, 0\}$  \\
\hline $\{24, 32, 64, 31, 0\}$\\
\hline $\{15, 20, 40, 46, 0\}$\\
\hline $\{36, 48, 25, 11, 0\}$\\
\hline
\end{tabular}
\end{center}
\end{ex}
\normalsize


Finally note that Theorem \ref{orthogonalcycles}(v) and the theorem of Dirichlet on arithmetic progressions
allow us to state the following.
\begin{thm}
Given any odd $k\geq3$ and any $r>0$, there are infinitely many values of $v$ for which there exists
a set of $r$ mutually orthogonal $k$-cycle systems of order $2v+1$.
\end{thm}


\section{Conclusions}

 We have established that there is an equivalence between a set of $r$ mutually orthogonal Heffter systems and
 a resolvable partial linear space of degree $r$ that we have called a Heffter space.

 We have determined various constructions for Heffter spaces of block size odd allowing us to claim
 that for any $r>0$ and any odd $k\geq3$, there are infinitely many values of $v$ for which there exists
a set of $r$ mutually orthogonal Heffter systems of order $v$ and block size $k$.

At this moment, reaching the analogous result for $k$ even appears to us quite difficult.
Anyway we are preparing a new paper \cite{BP3} with several constructions for Heffter spaces
with block size even and (an unfortunately) small degree $r$.

Several other open problems arise from this research.
In our opinion, the most intriguing is the one presented in Section 2.

\smallskip
{\bf Open problem.}\quad Does there exist a Heffter linear space?



\section*{Acknowledgements}
The authors are partially supported by INdAM - GNSAGA.

\end{document}